\documentclass{amsart}
\usepackage{amsmath, amsthm, amscd, amssymb,latexsym,enumerate,graphicx}
\usepackage{url}
\usepackage{array}
\usepackage[all]{xy}
\usepackage{enumitem}
\usepackage{hyperref}

\newfont{\cyrr}{wncyr10}

\newcommand{\Norm}{\operatorname{N}}
\newcommand{\Minspec}{\operatorname{Minspec}}
\newcommand{\Maxspec}{\operatorname{Maxspec}}

\newcommand{\id}{\operatorname{id}}
\newcommand{\coker}{\operatorname{coker}}

\def\WPic{\mathrm{WPic}}
\def\ee{e}
\def\m{{\mathfrak m}}
\def\zzzz{{p}}
\def\fff{{f}}

\def\L{\mathfrak{p}}
\def\Cl{\mathbf{Cl}}
\def\Z{\mathbb{Z}}
\def\Q{\mathbb{Q}}
\def\F{\mathbb{F}}
\def\R{\mathbb{R}}
\def\C{\mathbb{C}}
\def\MM{{\mathfrak{M}}}
\def\p{{\mathfrak{p}}}
\def\r{{\mathbf{r}}}

\def\ff{{\mathbf{f}}}
\def\cc{{I}}

\def\dd{{d}}
\def\h{{h}}

\def\SS{{\mathfrak{S}}}

\def\Q{{\mathbb Q}}

\def\Z{{\mathbb Z}}

\def\C{{\mathbb C}}
\def\R{{\mathbb R}}
\def\F{{\mathbb F}}
\def\N{\mathrm{N}}
\def\Spec{\mathrm{Spec}}

\def\Tr{\mathrm{Tr}}

\def\Rhom{\mathrm{Rhom}}

\def\Aut{\mathrm{Aut}}

\def\Hom{\mathrm{Hom}}

\def\fchar{\mathrm{char}}

\def\dim{\mathrm{dim}}

\def\lcm{\mathrm{lcm}}
\def\log{\mathrm{log}}
\def\det{\mathrm{det}}
\def\rk{\mathrm{rank}}
\def\rank{\mathrm{rank}}
\def\I{{\mathcal I}}

\def\O{{\mathcal O}}

\def\a{{\mathfrak a}}
\def\b{{\mathfrak b}}
\def\ccc{{\mathfrak c}}

\def\pile#1#2{\genfrac{}{}{0pt}{1}{#1}{#2}}

\def\F{{\mathbb F}}

\def\Tr{\mathrm{Tr}}

\def\trace{\mathrm{Tr}}
\def\tt{\mathrm{trace}}

\def\onto{\twoheadrightarrow}

\def\isom{\xrightarrow{\sim}}

\makeatletter
\newcommand{\bigperp}{%
  \mathop{\mathpalette\bigp@rp\relax}%
  \displaylimits
}

\newcommand{\bigp@rp}[2]{%
  \vcenter{
    \m@th\hbox{\scalebox{\ifx#1\displaystyle2.1\else1.3\fi}{$#1\perp$}}
  }%
}
\makeatother

\newtheorem{thm}{Theorem}
\numberwithin{thm}{section}

\newtheorem{lem}[thm]{Lemma}
\newtheorem{cor}[thm]{Corollary}
\newtheorem{prop}[thm]{Proposition}

\theoremstyle{definition}
\newtheorem{defn}[thm]{Definition}
\newtheorem{notation}[thm]{Notation}
\newtheorem{algorithm}[thm]{Algorithm}

\newtheorem{ex}[thm]{Example}
\newtheorem{exs}[thm]{Examples}
\newtheorem{rem}[thm]{Remark}
\newtheorem{rems}[thm]{Remarks}

\numberwithin{equation}{thm}

\title
[Testing isomorphism of lattices over CM-orders]
{Testing isomorphism of lattices over CM-orders}
\author[H.\ W.\ Lenstra, Jr.]{H.\ W.\ Lenstra, Jr.}
\address{Mathematisch Instituut, Universiteit Leiden, The Netherlands}
\email{hwl@math.leidenuniv.nl}
\author[A.\ Silverberg]{A.\ Silverberg}
\address{Department of Mathematics, University of California, Irvine, CA 92697}
\email{asilverb@uci.edu}
\thanks{Support for the research was provided by the Alfred P.~Sloan Foundation and the National Science Foundation.}

\begin{document}

\begin{abstract} 
A {\it CM-order\/} is a reduced order
equipped with an involution that mimics complex conjugation. The {\it Witt-Picard group\/} of such an order is a certain group of ideal classes that is closely related to the ``minus part'' of the class group. We present a deterministic polynomial-time algorithm for the following problem, which may be viewed as a special case of the principal ideal testing problem: given a CM-order, decide whether two given elements of its Witt-Picard group are equal. In order to prevent coefficient blow-up, the algorithm operates with lattices rather than with ideals. An important ingredient is a technique introduced by Gentry and Szydlo in a cryptographic context. Our application of it to lattices over CM-orders hinges upon a novel existence theorem for auxiliary ideals, which we deduce from a result of Konyagin and Pomerance in elementary number theory.
\end{abstract}

\maketitle

\section{Introduction}
\label{introsect}

An {\em order} is a commutative ring of which the additive group is isomorphic to
$\Z^n$ for some $n\in\Z_{\ge 0}$. 
We call $n$ the {\em $\Z$-rank} of the order.
In algorithms, we shall specify an order by a system
$(b_{ijk})_{i,j,k=1}^n$ of integers with the property that,
for some $\Z$-basis $\alpha_1,\ldots,\alpha_n$ of the order,
one has 
$\alpha_i \alpha_j = \sum_{k=1}^n b_{ijk} \alpha_k$ for all $1 \le i,j\le n$.

\begin{defn}
\label{CMorddef}
A {\em CM-order} $A$ is an order 
such that:
\begin{enumerate}
\item
$A$ has no non-zero nilpotent elements (i.e., $A$ is reduced),  
and
\item
$A$ is equipped with an 
automorphism $x \mapsto \bar{x}$ of $A$ such that 
$\psi(\bar{x})=\overline{\psi(x)}$
for all $x\in A$ and all ring homomorphisms $\psi : A \to \C$.
\end{enumerate}
\end{defn}

One can show that each CM-order has exactly  
one such automorphism, and 
it satisfies $\bar{\bar{x}} = x$ for all $x$
(see Lemma \ref{invollem} below).
In algorithms one specifies an automorphism of an order by means
of its matrix on the same $\Z$-basis $\alpha_1,\ldots,\alpha_n$ that
was used for the $b_{ijk}$.
\begin{exs}
\label{CMordexs}
Examples of CM-orders (see also Definition \ref{CMalgdef} and Examples \ref{CMordsexs}) 
include the following:
\begin{enumerate} 
\item
rings of integers of CM-fields (in particular, cyclotomic number fields),
\item
 group rings $\Z[G]$ for finite abelian
groups $G$, with $\bar{\sigma} = \sigma^{-1}$ for $\sigma\in G$,
\item
the rings $\Z\langle G\rangle = \Z[G]/(u+1)$ occurring in \cite{LwS},
where $G$ is a finite abelian group, $u\in G$ has order $2$,
and $\bar{\sigma} = \sigma^{-1}$ for $\sigma\in G$.
\end{enumerate}
\end{exs}

We show that CM-orders are easy to recognize. 
In Algorithm \ref{CMorddetalg} we give 
a deterministic polynomial-time algorithm that, given 
an order $A$, decides whether it has an automorphism that makes
it into a CM-order, and if so computes that automorphism.

Suppose $A$ is an order. We denote the $\Q$-algebra $A\otimes_\Z \Q$ by $A_\Q$.
We write $(A^+_\Q)_{\gg 0}$ for the set of all $w\in A_\Q$ with the
property that $\psi(w)\in\R_{>0}$ for each ring homomorphism
$\psi : A_\Q \to \C$; this is a subgroup of the group $A_\Q^\ast$ of units
of $A_\Q$.
By a {\em fractional $A$-ideal} we mean
a finitely generated sub-$A$-module $I$
of $A_\Q$ that spans $A_\Q$ as a $\Q$-vector space.
An {\em invertible} fractional $A$-ideal is a
fractional $A$-ideal $I$ such that there is a
fractional $A$-ideal $J$ with $IJ = A$,
where $IJ$ is the fractional $A$-ideal generated by
the products of elements from $I$ and $J$.

We next state our main result, which says that,
in a special case, principal ideal testing can be done in
polynomial time.

\begin{thm}
\label{mainIwthm}
There is a deterministic polynomial-time algorithm that given a CM-order $A$,
a fractional $A$-ideal $I$, and an element 
$w \in (A^+_\Q)_{\gg 0}$ satisfying $I\bar{I} = Aw$,
decides whether there exists $v\in A_\Q$ such that $I=Av$
and $v\bar{v}=w$, and if so computes such an element $v$.
\end{thm}

More generally, we show:

\begin{thm}
\label{mainIwthmcor}
There is a deterministic polynomial-time algorithm that given a CM-order $A$,
fractional $A$-ideals $I_1$ and $I_2$, and elements 
$w_1,w_2 \in (A^+_\Q)_{\gg 0}$ satisfying $I_1\overline{I_1} = Aw_1$
and $I_2\overline{I_2} = Aw_2$,
decides whether there exists $v\in A_\Q$ such that $I_1=vI_2$
and $w_1 =v\bar{v}w_2$, and if so computes such an element $v$.
\end{thm}

See the very end of this paper for proofs of 
Theorems \ref{mainIwthm} and \ref{mainIwthmcor}.

The set of all pairs $(I,w)$ as in Theorem \ref{mainIwthm} is a
multiplicative group (see Section \ref{WPsect}), and
$
\{ (Av,v\bar{v}) : v\in A_\Q^\ast \}
$
is a subgroup.
Writing $\WPic(A)$ for the quotient group, 
Theorem \ref{mainIwthmcor} provides an
efficient equality test in $\WPic(A)$.
The set of principal invertible fractional $A$-ideals
$
\{ Av : v\in A_\Q^\ast \}
$
is a subgroup of the set of all invertible fractional $A$-ideals;
write $\Cl(A)$ for the quotient group,
and write
$\Cl^-(A)$ for the subgroup of classes $[I]\in\Cl(A)$
for which 
$I\bar{I}$ is principal.
We can show that the group homomorphism 
$\WPic(A) \to \Cl^-(A)$ sending the class of $(I,w)$
to the class of $I$ is almost an isomorphism in the sense
that both its kernel and its cokernel are annihilated by $2$
(Theorem \ref{wpickercoker} below).
Hence we can efficiently do an equality test in a group that is
closely related to the ``minus part'' of the class group of a
CM-order.

To obtain these results, we view our fractional $A$-ideals as 
lattices with an $A$-module structure. 
This allows us to avoid blow-up of the coefficients with respect to
a $\Z$-basis, when ideals are repeatedly multiplied together.

By a {\em lattice}, or {\em integral lattice}, 
we mean a finitely generated free abelian group $L$ 
equipped with a positive definite symmetric
$\Z$-bilinear map
$\langle \, \cdot \, , \,  \cdot \,  \rangle : L \times L \to \Z$;
this map will be referred to as the {\em inner product}.
A lattice is specified by means of the matrix
$(\langle b_i,b_j \rangle)_{i,j=1}^m$ for some $\Z$-basis $b_1,\ldots,b_m$ of $L$.

Let $A$ be a CM-order. 
By an {\em $A$-lattice} we mean a lattice $L$ that is
given an $A$-module structure with the property that for all
$a\in A$ and $x,y\in L$ one has
$\langle ax,y \rangle = \langle x,\bar{a}y \rangle$.
One specifies an $A$-lattice by specifying it as a lattice and
listing the system of $nm^2$ integer coefficients that
express $\alpha_ib_j$ on $b_1,\ldots,b_m$, with the
$\Z$-bases $(\alpha_i)_{i=1}^n$ for $A$ and $(b_j)_{j=1}^m$ for $L$
being as above.
An {\em $A$-isomorphism} $f: L \to M$ of $A$-lattices is an
isomorphism of $A$-modules with 
$\langle f(x),f(y) \rangle = \langle x,y \rangle$ for all
$x,y\in L$; such an isomorphism is specified by its
matrix on the $\Z$-bases for $L$ and $M$ that are used.
An example of an $A$-lattice is the $A$-module $A$ itself,
with inner product $( a,b ) = \Tr(a\bar{b})$;
here $\Tr : A \to\Z$ is the trace function of $A$ as a
$\Z$-algebra. This $A$-lattice is called the {\em standard}
$A$-lattice. 

Deciding whether two lattices are isomorphic is a notorious algorithmic problem. 
Our results here and in \cite{LwS} show that the problem admits a 
polynomial-time solution if the lattices are equipped with sufficient structure.

\begin{thm}
\label{mainthm}
There is a deterministic polynomial-time algorithm that, given a CM-order $A$
and an $A$-lattice $L$, decides whether or not $L$ is $A$-isomorphic
with the standard $A$-lattice, and if so, computes such an $A$-isomorphism.
\end{thm}

The algorithm and the proof are given in Section \ref{mainalgorsect}.
An imprecise summary is as follows. 
Finding an $A$-isomorphism as in Theorem \ref{mainthm} is equivalent to
finding a ``short'' vector in $L$.
Using a suitable tensor power $L^m$, one can force a short vector
to lie in a certain coset of $L^m$ modulo $\a L^m$.
Here $\a$ is an auxiliary ideal of $A$ that is chosen to have large norm,
which enables us to recover the short vector itself.
If one can do this for $m_1$ and $m_2$, then they combine into a
short vector in $L^{\gcd(m_1,m_2)}$.
Ultimately, one obtains a short vector in $L^k$, where the final $\gcd$ $k$
has relatively small prime factors.
Removing these one by one, one finds the desired short vector in $L$.

As a corollary of Theorem \ref{mainthm} we obtain the following result (with
{\em invertible} defined as in Definition \ref{Alatticeinvertdef}), from which
Theorem \ref{mainIwthmcor} follows.
\begin{thm}
\label{mainthmcor}
There is a deterministic polynomial-time algorithm that given a CM-order $A$
and invertible $A$-lattices $L$ and $M$, decides whether or not $L$
and $M$ are isomorphic as $A$-lattices, and if so, exhibits 
such an $A$-isomorphism.
\end{thm}

Theorems \ref{mainthm} and \ref{mainthmcor} generalize the main results of \cite{LwS},
which concerned the special case $A = \Z\langle G\rangle$ mentioned 
in Example \ref{CMordexs}(iii).
While the proofs are  
different from those in \cite{LwS}, since the general strategies are similar 
we structured this paper so that in broad outline our
proofs  
follow the same logical order as that of \cite{LwS}, which was
devoted to the case $A = \Z\langle G\rangle$.

One important difference between the present paper and   \cite{LwS}
lies in the manner in which auxiliary ideals of $A$ are constructed.
In the case $A = \Z\langle G\rangle$, we could use
Linnik's theorem for this purpose (see Section 18 of \cite{LwS}),
but for general $A$ this cannot be done.
Here we show that the following result suffices.

\begin{thm}
\label{newLinnik}
Let $A$ be an order of $\Z$-rank $n \ge 1$, and let $\ell$
be a prime number with $\ell > n^2$.
Then there exists a maximal ideal $\p$ of $A$ that contains
a prime number $p \le 4(1+(\log \, n)^2)$
and that satisfies $\#(A/\p) \not\equiv 1\bmod{\ell}$.
\end{thm}

It is remarkable that the upper bound $4(1+(\log \, n)^2)$ on $p$ in 
Theorem \ref{newLinnik} depends on $A$ only through its $\Z$-rank $n$, 
and that it is so small. One may actually conjecture that Theorem \ref{newLinnik} 
remains true with $4(1+(\log \, n)^2)$ replaced by $5$;
we give a heuristic argument after the proof of Proposition \ref{Hope1} below. 
For the elementary proof of Theorem \ref{newLinnik}, 
see the proof of Proposition \ref{Hope1}, which relies on
a result of Konyagin and Pomerance \cite{KP}.

The price that we pay for the very small upper bound on $p$
in Theorem \ref{newLinnik} 
is that we have to work with ideals $\a$ of $A$ that are
not necessarily generated by elements of $\Z$.
This leads to a number of technical difficulties
(see for example Sections \ref{lowerbdssect}, \ref{Hope2sect}, \ref{Hope2sectbis},
and  \ref{applyauxprssect}) 
that were not present in \cite{LwS}.
Applying Theorem \ref{newLinnik} instead of Linnik's theorem 
in the case $A = \Z\langle G\rangle$, one may expect
to obtain a dramatically lower run time exponent than the
one achieved in \cite{LwS}.

Another difference between this paper and \cite{LwS} is that,
in order to preserve integrality, we replaced
the ``scaled trace map'' $t$ (from Definition 6.2 of  \cite{LwS})
by the trace map $\Tr$ given before Theorem \ref{mainthm}.
As a consequence, the inner product $( \, \, \, , \,\, \,  )$
used for the standard $A$-lattice in this paper is,
in the special case $A = \Z\langle G\rangle$, equal to $n$
times the inner product used in \cite{LwS}, where
$n = (\# G)/2$. 
For similar reasons, the definition of
an {\em invertible} $A$-lattice (see Definition \ref{Alatticeinvertdef}) requires
more care than in  \cite{LwS}. 
We needed to redefine {\em short vector} (Definition \ref{shortdef}),
and the short vectors now behave differently.
What remains true is that an $A$-lattice is $A$-isomorphic to the
standard $A$-lattice if and only if it is invertible and has a short vector.
However, the group of roots of unity in $A$ now 
might be too large to even write down in polynomial time, so the
set of short vectors in $L$ and thus the set of 
all $A$-isomorphisms from $L$ to $A$ might be too large to enumerate.

Any choices and recommendations that we make, especially those concerning the selection of auxiliary ideals, are intended to optimize the efficiency of our proofs rather than of our algorithm.

Our work on this subject was inspired by an algorithm
of Gentry and Szydlo (Section 7 of \cite{GS}), and is related to our work on
lattices with symmetry \cite{LenSil,LwS}. 
In this paper we give 
the details for the proofs of the
results announced in our 2013 workshop on this subject
\cite{2013workshop}; see especially \cite{CMordersVideo}. 
In \cite{Kirchner}, P.\ Kirchner gave a 
version of our Theorem \ref{mainIwthm} that,
due to the inapplicability of Linnik's theorem for general CM-orders, 
either assumes the generalized Riemann hypothesis
or allows probabilistic algorithms.

The setting in this paper is applicable to the setting
considered by Garg, Gentry, and Halevi in \cite{GGH} 
where the CM-order $A$ is a cyclotomic ring $\Z[\zeta_m]$,
to the setting considered by Gentry and Szydlo where the order is
$\Z[X]/(X^m-1)$, and to the orders $\Z[X]/(X^m+1)$ used
for fully homomorphic encryption.

\subsection{Overview of algorithm for Theorem \ref{mainthm}}
The algorithm starts by testing whether the given $A$-lattice $L$ 
is invertible. Then it computes the primitive idempotents of $A$, 
in order to decompose $A$ as a product of connected rings and
reduce the problem to the case where $A$ is connected.
We work in a $\Z$-graded extended tensor algebra
$\Lambda = \bigoplus_{i\in\Z} L^{\otimes i}$.
Let $n = \rk_\Z(A)$.
We make use of Theorem \ref{newLinnik}
to construct a finite set of ``good'' ideals $\a$ of $A$, and
for each $\a$ a positive integer $k(\a)$ divisible by the exponent of
the group $(A/\a)^\ast$, such that every prime divisor of $k = \gcd \{ k(\a) \}$
is at most $n^2$.
Next, for each good ideal $\a$ one tries to find
a short vector $z_\a\in L^{\otimes k(\a)}$
such that for every short vector $z$ of $L$ one has 
$z^{\otimes k(\a)} = z_\a$; 
if this fails, one concludes that $L$ is not $A$-isomorphic to the
standard $A$-lattice (and terminates).
We then use the Euclidean algorithm to construct from the $z_\a$ a vector
$w\in L^{\otimes k}$ such that if $L$ has a short vector $z$ then $z^{\otimes k}=w$.
If $p_1,\ldots, p_m$ are the prime divisors of $k$ with multiplicity,
we use our results on graded orders from \cite{Kronecker}
and our results on roots of unity in orders from \cite{RoU} 
to either obtain a short vector $z_1$ in $L^{\otimes k/p_1}$, then
a short vector $z_2$ in $L^{\otimes k/(p_1p_2)}$, and so on,
until one obtains a short vector in $L$, or else prove that
$L$ has no short vector.
If the algorithm produces a short vector $z$ in $L$,
then the map $A \to L$, $a\mapsto a z$ is an $A$-isomorphism,
and otherwise no $A$-isomorphism exists.

\subsection{Structure of the paper}
In Sections  \ref{CMalgssect}--\ref{Alatticesect} 
we give background and results about
CM-orders and 
$A$-lattices.
In Section \ref{redbasautsect}
we obtain bounds for LLL-reduced bases of
invertible lattices (Proposition \ref{LLLlem})
that allow us to show that the Witt-Picard group is finite
and that our algorithms run in polynomial time.
In Section \ref{latticecosetssect}
we show how to find the unique ``short'' vector in a
suitable lattice coset, when such a vector exists. 
In Section \ref{regsect}
we characterize short vectors in $A$-lattices.
In Section \ref{lowerbdssect} we give conditions under
which we can easily apply the results in Section \ref{latticecosetssect}.
In Section \ref{ideallatsect}
we relate $A$-lattices to fractional $A$-ideals,
and in Section \ref{invAlatsect} we give results
on invertible $A$-lattices.
In Section \ref{shorttensorsect} we study short vectors
in invertible $A$-lattices; in particular, we show that
an $A$-lattice is $A$-isomorphic to the
standard one
if and only if it is invertible and has a short vector.
In Section \ref{WPsect} we study
the Witt-Picard group of $A$.
Section \ref{multexpinvlatsect} deals with 
multiplying and exponentiating invertible $A$-lattices.
In Section \ref{tensorsect} we introduce the extended tensor algebra $\Lambda$,
which is a single algebraic structure
that comprises all rings and lattices occurring in our main algorithm.
Sections \ref{Hope2sect} and \ref{Hope2sectbis}
are the heart of the paper, and consist of
finding the auxiliary ideals.
In Section \ref{applyauxprssect} we give algorithms that
make use of our choice of auxiliary ideals; we use
these algorithms as subroutines for 
our main algorithm, which is given in Section \ref{mainalgorsect}.

\subsection{Notation}
As usual, $\Z$, $\Q$, $\R$, and $\C$ denote respectively the 
ring of integers, and fields of rational numbers,
real numbers, and complex numbers.
Suppose $B$ and $C$ are commutative rings.
Let $\Rhom(B,C)$ denote the
set of ring homomorphisms from $B$ to $C$, 
let $\Spec(B)$ denote the set of prime ideals of $B$,
and
let $\mu(B)$ denote the group of roots of unity of $B$.
If $\p\in\Spec(B)$, let $B_\p$ denote the localization of $B$ at $\p$
and let $\N(\p) = \#(B/\p)$.
If  $A$ is an order, 
let $\Minspec(A)$ denote the set of minimal prime
ideals of $A$ and let $\Maxspec(A)$ denote the set of maximal 
ideals of $A$.
If $R$ is a commutative ring and $B$ and $C$ are $R$-algebras,
let $\Rhom_R(B,C)$
denote the set of $R$-algebra homomorphisms from $B$ to $C$, and
if $D$ is a $\Z$-module let 
$D_R = D\otimes_\Z R$.

\subsection*{Acknowledgments}
We thank 
all the participants of the 2013 Workshop on Lattices with Symmetry,
and especially Daniele Micciancio for his interest in our work.
We also thank Abtien Javanpeykar for his help with Theorem \ref{newLinnik}.

\section{CM-fields and CM-algebras}
\label{CMalgssect}

By a {\em classical CM-field} we 
will mean a totally imaginary quadratic extension of a
totally real number field. 
We define a {\em CM-field} to be any subfield of 
a classical CM-field.
A number field is a CM-field if and only if it is either a classical CM-field or totally real (by Lemma 18.2(iv) on p.~122 of \cite{ShTan}).

\begin{defn}
\label{CMalgdef}
A {\em CM-algebra} is a commutative $\Q$-algebra $E$ such that:
\begin{enumerate}
\item
$\dim_\Q(E) < \infty$,
\item
$E$ has no non-zero nilpotent elements,
\item
$E$  is equipped with an automorphism $x \mapsto \bar{x}$ 
such that $\psi(\bar{x})=\overline{\psi(x)}$
for all $x\in E$ and all 
$\psi \in\Rhom(E,\C)$.
\end{enumerate}
\end{defn}

\begin{rem}
\label{CMalgrems}
It follows from Lemma 18.2(i) on p.~ 122 of \cite{ShTan} that 
a finite dimensional commutative $\Q$-algebra 
$E$ is a CM-algebra if and only if 
all elements of $E$ are separable and
$E/\m$ is a CM-field for all $\m\in\Spec(E)$.
In other words, a finite dimensional commutative $\Q$-algebra
is a CM-algebra if and only if it is a product of finitely many
CM-fields.
In particular, the CM-algebras that are fields are exactly the CM-fields.
\end{rem}

\begin{rem}
\label{CMalgpos}
If $E$ is a CM-algebra and $x\in E$, then 
$\Tr_{E/\Q}(x\bar{x})>0$ for all $x\in E\smallsetminus \{0\}$.
\end{rem}

\begin{lem}
\label{QRquadformlem}
Suppose $V$ is a finite-dimensional $\Q$-vector space, $f : V \to \Q$ is
a quadratic form, and $f_\R : V_\R \to \R$ is the $\R$-linear extension of $f$.
Then $f$ is positive definite if and only if $f_\R$ is positive definite.
\end{lem}

\begin{proof}
Diagonalize $f$ over $\Q$, so $f(x) = \sum_{i=1}^n a_ix_i^2$ 
where the $x_i$ are the coordinates of $x$ on some $\Q$-basis of $V$ and all $a_i \in\Q$.
Then $f$ is positive definite if and only if all $a_i>0$.
Using the same basis for $V_\R$ over $\R$ now gives the desired result.
\end{proof}

The following result will be used to prove Proposition \ref{Bgraded}.
It generalizes Lemma 2 on p.~37 of \cite{ShTan}, which dealt with the
case where $E$ is a number field.

\begin{prop}
\label{involCMprop}
Suppose $E$ is a finite dimensional commutative $\Q$-algebra, 
$\rho\in\Aut(E)$, 
and $\Tr_{E/\Q}(x\rho(x))>0$ for all $x\in E\smallsetminus \{0\}$.
Then:
\begin{enumerate}
\item
$\Tr_{E_\R/\R}(x\rho(x)) > 0$ for all $x\in E_\R \smallsetminus  \{0\}$,
\item 
$\rho(\rho(x))=x$ for all $x\in E$,
\item
and $E$ is a CM-algebra with $\rho$ serving as $\bar{\,\,}$.
\end{enumerate}
\end{prop}

\begin{proof}
By Lemma \ref{QRquadformlem} we have (i).

If $y$ is a nilpotent element of $E$, then
$y\rho(y)$ is nilpotent, so $\Tr_{E/\Q}(y\rho(y))=0$,
so $y=0$ by our hypothesis. Thus, $E$ is reduced.

We have $E \hookrightarrow E_\R 
= \R^r \times \C^s$ for some $r,s\in\Z_{\ge 0}$,
and $\rho$ extends to an automorphism of $E_\R$ as an $\R$-algebra.
For $1 \le j \le r+s$, let 
$\alpha_j = (0,\ldots,0,1,0,\ldots,0) \in \R^r \times \C^s=E_\R$ 
with $1$ in the $j$-th position.
We claim that $\rho({\alpha_j})=\alpha_j$ for all $j$.
If not, then since the $\alpha_j$'s are exactly the primitive idempotents of $E_\R$ 
we have $\rho({\alpha_j})=\alpha_k$ for some $k\neq j$, so
$0 < \Tr_{E_\R/\R}(\alpha_j\rho(\alpha_j)) = 
\Tr_{E_\R/\R}(\alpha_j\alpha_k) = 0$, a contradiction.
Thus $\rho$ acts componentwise, and is the identity on each
$\R$ and either the identity or complex conjugation on each $\C$.
In particular, $\rho(\rho(x))=x$ for all $x\in E_\R$, and we have (ii).

If $\rho$ is the identity on the $j$-th $\C$, then
letting $x=\sqrt{-1}\alpha_j$ we have 
$$
\Tr_{E_\R/\R}(x\rho(x))=\Tr_{E_\R/\R}(-\alpha_j) = -2 < 0,
$$ 
a contradiction.
It follows that $\psi(\rho(x))=\overline{\psi(x)}$
for all $\psi \in\Rhom(E,\C)$ and all 
$x\in E$, giving (iii).
\end{proof}

The next algorithm will be used in Algorithm \ref{CMorddetalg}.
For the input, 
a degree $n$ field $F$ is specified
(as in \cite{Qalgs})
by listing a
system of ``structure constants'' $a_{ijk}\in\Q$, for $i,j,k\in\{ 1,2,\ldots,n\}$,
that determine the multiplication in the sense that for some
$\Q$-basis $\{\alpha_1,\alpha_2,\ldots,\alpha_n\}$ of $F$ 
one has $\alpha_i\alpha_j = \sum_{k=1}^n a_{ijk}\alpha_k$ for all $i,j$.
Elements of $F$ are then represented by their vector of coordinates
on that basis.

\begin{algorithm}
\label{CMtestalgor}
Given a number field $F$, the algorithm decides whether $F$ is
a CM-field, and if so computes $\bar{\,\,} \in\Aut(F)$.

Steps:
\begin{enumerate}
\item
Compute $\Aut(F)$.
\item
For all $\sigma\in\Aut(F)$ with $\sigma^2=\id_F$ in succession
compute $\Tr_{F/\Q}(\alpha_i\cdot \sigma(\alpha_j))$ for the given
$\Q$-basis $\{ \alpha_1,\ldots,\alpha_n\}$ of $F$ and test whether
for all $k\in\{ 1,2,\ldots,n\}$ we have
$\det((\Tr_{F/\Q}(\alpha_i\cdot \sigma(\alpha_j)))_{i,j=1}^k) > 0$.
If not, pass to the next $\sigma$ or if there is no next $\sigma$
terminate with ``no''.
If yes, terminate with ``yes'' and $\bar{\,\,} = \sigma$.
\end{enumerate}
\end{algorithm}

\begin{prop}
\label{CMtestpf}
Algorithm \ref{CMtestalgor} is correct and runs in  polynomial time.
\end{prop}

\begin{proof}
Let $f_\sigma : F \to \Q$ be the quadratic form
$f_\sigma(x) = \Tr_{F/\Q}(x\sigma(x))$.
Then $f_\sigma$ is positive definite if and only if $(f_\sigma)_\R$ is positive definite,
by Lemma \ref{QRquadformlem}.
Further, $(f_\sigma)_\R$ is positive definite if and only if the matrix
$A = (\Tr_{F/\Q}(\alpha_i\cdot \sigma(\alpha_j)))_{i,j=1}^n$ is positive definite.
By Sylvester's criterion, $A$ is positive definite if and only if
its leading principal minors $\det((\Tr_{F/\Q}(\alpha_i\cdot \sigma(\alpha_j)))_{i,j=1}^k)$
are all positive.
Correctness of the algorithm now follows from Proposition \ref{involCMprop} and Lemma \ref{CMalgpos}.
Computing $\Aut(F)$ can be done in polynomial time, by
\S 2.9 of \cite{AlgNoThy}.
\end{proof}

\begin{rem}
\label{CMalgalgorrem}
There is a deterministic polynomial-time algorithm that
given a finite dimensional commutative $\Q$-algebra 
$E$ decides whether it is a CM-algebra and if so produces $\bar{\,\,}$.
Namely, use Algorithms 5.5 and 7.2 of \cite{Qalgs} to determine whether 
all elements of $E$ are separable and 
if so to compute all $\m\in\Spec(E)$ and apply
Algorithm \ref{CMtestalgor} above to check whether each $E/\m$ is
a CM-field and find its automorphism $\bar{\,\,}$.
\end{rem}

\section{CM-orders}
\label{CMorderssect}

If $A$ is a reduced order,
then the trace map  
$\Tr = \Tr_{A/\Z} : A \to \Z$
extends by linearity to trace maps 
$\Tr : A_\Q \to \Q$ and $\Tr : A_\R \to \R$,
and for all $a\in A$ we have $\Tr(a) = \sum_{\psi\in\Rhom(A,\C)} \psi(a)$. 
(Note that $\#\Rhom(A,\C) = \rk_\Z(A)$.) 

Recall that the discriminant $\Delta_{A/\Z}$ of an order $A$
is the determinant of the matrix 
$(\Tr_{\O/\Z}(\alpha_i\alpha_j))_{i,j}$ for any $\Z$-basis $\{\alpha_i\}$ of $A$.

In Section \ref{introsect}, a CM-order $A$ was specified 
by $n=\rk_\Z(A)$,
and a system $(b_{ijk})_{i,j,k=1}^n$ of integers such that
for some $\Z$-basis $\{\alpha_i\}_{i=1}^n$ of $A$ one has 
$\alpha_i \alpha_j = \sum_{k=1}^n b_{ijk} \alpha_k$ for all $1 \le i,j\le n$,
and a matrix giving $\bar{\,\,}$ on $A$.
We improve the way the data for $A$ are specified, as follows.
Note that $\Tr(\alpha_i) = \sum_{j=1}^n b_{ijj}$.
It is straightforward to use the specified data
to compute the Gram matrix $((\alpha_i,\alpha_j))_{1\le i,j\le n}$
for $A$ relative to the basis $\{ \alpha_i\}_{i=1}^n$,
where $(a,b) = \Tr_{A/\Z}(a\bar{b})$ for all $a,b\in A$,
and compute $\det((\alpha_i,\alpha_j)) = |\Delta_{A/\Z}|$, which is
the determinant of $A$ as a lattice (Definition \ref{detlatdef} below).
Run the LLL lattice basis reduction algorithm (\cite{LLL}) to replace
$\{ \alpha_i\}_{i=1}^n$ by an LLL-reduced basis
(see Definition \ref{LLLreduceddefn} for the definition), and recompute the constants
$b_{ijk}$ and the matrix giving $\bar{\,\,}$.
We always first run the above algorithm to give an LLL-reduced
basis, and convert back to the original basis at the end.
We suppress this in the algorithms below, and assume our input $A$ 
is given with an LLL-reduced basis, and that we have kept track
of how the LLL-basis is expressed in terms of the original basis
$\{ \alpha_i\}$, so that one can give the final answer in terms
of the original basis.

\begin{lem}
\label{KalglemBourb}
If $A$ is a reduced order, then
$
\bigcap_{\psi \in \Rhom(A,\C)}\ker(\psi) = 0.
$
\end{lem}

\begin{proof}
Let $n=\rk_\Z(A)$. Since $A$ is reduced, we have
$A \subset A_\C \cong \C^n$,
so 
$\bigcap_{\psi \in \Rhom(A,\C)}\ker(\psi) \subset 
\bigcap_{\psi \in \Rhom_\C(\C^n,\C)}\ker(\psi) = 0.
$
\end{proof}

\begin{defn}
\label{innerproddef}
If $A$ is a CM-order, and $a,b\in A$, define
$(a,b) = \Tr_{A/\Z}(a\bar{b})$.
\end{defn}

\begin{lem}
\label{ordlatlem}
If $A$ is a CM-order, then $A$ is an integral lattice with respect to the inner product
$(  \,\, , \,\, )$.
\end{lem}

\begin{proof}
The map 
$(a,b) \mapsto \Tr_{A/\Z}(a\bar{b})$
is clearly $\Z$-valued, $\Z$-bilinear, and symmetric.
If $a\in A$, then 
$\psi(a\bar{a}) = \psi(a){\overline{\psi(a)}} \in \R_{\ge 0}$
for all $\psi\in\Rhom(A,\C)$,
so 
$$(a,a) = \Tr_{A/\Z}(a\bar{a}) = \sum_{\psi\in\Rhom(A,\C)} \psi(a\bar{a}) \in \R_{\ge 0}.$$
Suppose $a\neq 0$.
Since $\bigcap_{\psi\in\Rhom(A,\C)} \ker \psi =0$ by Lemma \ref{KalglemBourb},
there exists $\psi\in\Rhom(A,\C)$ such that $\psi(a) \neq 0$.
Thus $\psi(\bar{a}) =\overline{\psi(a)} \neq 0$, so
$\psi(a\bar{a}) = \psi(a)\psi(\bar{a}) \neq 0$, so $(a,a) > 0$.
\end{proof}

\begin{lem}
\label{invollem}
Suppose $A$ is a CM-order. Then:
\begin{enumerate}
\item 
$a \mapsto \bar{a}$ is an involution on $A$ (i.e., $\bar{\bar{a}} = a$
for all $a\in A$);
\item
$A$ has exactly one involution satisfying Definition \ref{CMorddef}(ii);
\item 
the involution $\bar{\,\,}$ extends $\R$-linearly to $A_\R$,  
and is the unique
involution on $A_\R$ 
such that $\psi(\bar{a})=\overline{\psi(a)}$ for all $a\in A_\R$ and
all $\psi \in \Rhom_\R(A_\R,\C)$; 
\item 
$\Tr_{A_\R/\R}(a\bar{a}) > 0$ for all non-zero $a\in A_\R$.
\end{enumerate}
\end{lem}

\begin{proof}
For all 
$\psi\in\Rhom(A,\C)$ and  all $a\in A$
we have 
$\psi(a)=\overline{\psi(\bar{a})}=\psi(\bar{\bar{a}})$, 
so $a=\bar{\bar{a}}$ by Lemma \ref{KalglemBourb}.

Suppose $\rho_1$ and $\rho_2$ are two involutions satisfying Definition \ref{CMorddef}(ii).
Then for all $a\in A$ and all $\psi\in\Rhom(A,\C)$ we have
$\psi(\rho_1(a)) = \overline{\psi(a)} = \psi(\rho_2(a))$.
Thus 
$\rho_1=\rho_2$ by Lemma \ref{KalglemBourb}, giving (ii).

The map  $\bar{\,\,}$ extends $\R$-linearly to $A_\R$, and the
proofs of (i) and (ii) extend to $A_\R$ to give (iii).

We have $A_\R \cong \R^r \times \C^s$ for some $r,s\in\Z_{\ge 0}$,
and $\Rhom_\R(A_\R,\C)=\{\psi_j\}_{j=1}^{r+2s}$ 
with
$\psi_j: A_\R \to \R$ for $1 \le j \le r$ and
$\psi_{s+j}=\overline{\psi_{j}}$ for $r+1\le j\le r+s$.
For (iv), suppose $0 \neq a\in A_\R$.
Then 
$$
\Tr_{A_\R/\R}(a\bar{a}) = \sum_{\psi\in\Rhom_\R(A_\R,\C)} \psi(a\bar{a})
 = \sum_{i=1}^r \psi_i(a)^2 + 2\sum_{i=r+1}^{r+s} \psi_i(a)\overline{\psi_i(a)} > 0.
$$
\end{proof}

\begin{rem}
\label{CMordiffresult}
If $A$ is an order, then $A$ is a CM-order if and only if $A_\Q$ is
a CM-algebra and $A = \bar{A}$.
\end{rem}

\begin{defn}
\label{Ahatetcdef}
If $A$ is a CM-order, define 
$$
\hat{A} = \{ a\in A_\Q : \Tr_{A_\Q/\Q}(aA) \subset \Z \} \subset A_\Q,
$$
$$
A_\R^+ = \{ a\in A_\R : a=\bar{a} \} 
= \{ a\in A_\R : \forall \psi\in\Rhom(A,\C), \psi(a)\in\R \}, 
$$
\begin{align*}
(A_\R^+)_{>0} &= \{ a\in A_\R : \forall \psi\in\Rhom_\R(A_\R,\C), \psi(a)\in\R_{\ge 0} \text{ and } 
\exists \psi : \psi(a)> 0 \} \\
&= \{ a\in A_\R : \forall \psi\in\Rhom_\R(A_\R,\C), \psi(a)\in\R_{\ge 0} \} - \{ 0 \}, \\
(A_\R^+)_{\gg 0} &= \{ a\in {A_\R} : \forall \psi\in\Rhom_\R(A_\R,\C), \psi(a)\in\R_{> 0} \},
\end{align*}
and for $B \subset A_\R$ define 
$$
B^+ = B \cap A_{\R}^+, \quad 
B^+_{>0} = B \cap  (A_\R^+)_{>0}, \quad 
B^+_{\gg 0} = B \cap (A_\R^+)_{\gg 0}.
$$
\end{defn}

We will apply 
Definition \ref{Ahatetcdef}
with $B = A$ and with $B = \hat{A}$.

The set $A^+_{>0}$ is not necessarily closed under multiplication (since $A$ is not
necessarily a domain).

\begin{exs}
\label{CMordsexs}
\begin{enumerate}[leftmargin=*]
\item
If $F$ is a CM-field,
then the ring of integers of $F$
is a CM-order, with complex conjugation serving as $\bar{\,\,}$.
\item
If $B$ is a subring of a CM-order, then the subring generated
by $B$ and $\bar{B}$ is a CM-order.
\item
If $A_1$ and $A_2$ are CM-orders, then so are
$A_1 \times A_2$ and $A_1\otimes_\Z A_2$.
\item
Suppose $G$ is a finite abelian group of order $n$.
If $A = \Z[G]$ then $\hat{A} = \frac{1}{n}\Z[G]$.
\end{enumerate}
\end{exs}

\begin{ex}
A suborder of a CM-order is not necessarily a CM-order,
since the automorphism  $\bar{\,\,}$ might not preserve the suborder.
For example,
suppose $A$ is a CM-order, and $\m$ is a maximal ideal of $A$
such that $\m\neq\bar{\m}$ and $A/\m$ is not a prime field.
Then $A/\m$ contains a prime field $F$, and the inverse image of 
$F$ under the natural map $A \to A/\m$ is a proper subring $R$ of
$A$ such that $R\neq \bar{R}$, so $R$ is not a CM-order.
\end{ex}

\begin{ex}
Suppose that $q$ is a prime power and $\pi$ is a $q$-Weil number, i.e.,
$\pi$ is an algebraic integer in $\C$ such that 
$|\sigma(\pi)| =\sqrt{q}$ for all $\sigma\in\Aut(\C)$.
Then $\Z[\pi,\bar{\pi}]$ is a CM-order, but if $[\Q(\pi):\Q] > 2$
then its suborder $\Z[\pi]$ is not a CM-order.
To see the latter, consider the irreducible polynomial $\sum_{i=0}^n a_iX^i \in\Z[X]$
that $\pi$ satisfies with $a_n = 1$. Then 
$
\pi\sum_{i=0}^{n-1} a_{i+1}\pi^i = -a_0 = \pm q^{n/2} = \pm q^{n/2-1} \pi\bar{\pi}.
$
Thus, $\bar{\pi} = \pm q^{1-n/2}(\sum_{i=0}^{n-1} a_{i+1}\pi^i)$.
The coefficient of $\bar{\pi}$ at $\pi^{n-1}$ is $\pm q^{1-n/2} \not\in \Z$,
so $\bar{\pi} \not\in \Z[\pi]$.
The order $\Z[\pi]$ passes steps (i)--(iv) of Algorithm \ref{CMorddetalg} below, but not step (v).
\end{ex}

\begin{prop}
\label{1equiv}
Suppose $A$ is a CM-order
and $a\in {A}^+_{\gg 0}$. 
Then the following are equivalent:
\begin{enumerate}
\item
$a=1$,
\item
$\Tr(a) = \rank_\Z(A)$, 
\item
$\Tr(a) \le \rank_\Z(A)$. 
\end{enumerate}
\end{prop}

\begin{proof}
The implications (i) $\Rightarrow$ (ii) $\Rightarrow$ (iii) 
are clear.
Since $a\in {A}^+_{\gg 0}$,
we have $\sigma(a) \in\R_{> 0}$ for all $\sigma\in\Rhom(A_\Q,\C)$,
and $\prod_\sigma \sigma(a) \in\Z_{>0}$.
Assuming (iii),
then applying the arithmetic-geometric mean inequality we have
\begin{multline*}
\rank_\Z(A) 
\ge \Tr(a)
= \sum_{\sigma\in\Rhom(A_\Q,\C)} \sigma(a)
= \rank_\Z(A) \cdot\frac{\sum_{\sigma} \sigma(a)}{\#\Rhom(A_\Q,\C)} \\
\ge \rank_\Z(A) \cdot[\prod_{\sigma\in\Rhom(A_\Q,\C)} \sigma(a)]^{1/\#\Rhom(A_\Q,\C)}
\ge \rank_\Z(A).
\end{multline*}
Thus we have equality everywhere, and all $\sigma(a)=1$, so $a=1$, and
(iii) $\Rightarrow$ (i).
\end{proof}

The following algorithm is patterned after the algorithm described in 
Remark~\ref{CMalgalgorrem}.

\begin{algorithm}
\label{CMorddetalg}
Given an order $A$, the algorithm decides whether $A$ is a CM-order, 
and if so computes the automorphism $\bar{\,\,}$.

Steps:
\begin{enumerate}
\item
Compute 
the discriminant $\Delta_{A/\Z}$ of $A$.
If it is $0$,
terminate with ``no''.
\item
Use Algorithm 7.2 of \cite{Qalgs} to find all $\m\in\Spec(A_\Q)$
and to find a $\Q$-basis for each field $A_\Q/\m$. 
\item
For each $\m\in\Spec(A_\Q)$, 
apply Algorithm \ref{CMtestalgor} to determine whether the field
$A_\Q/\m$ is a CM-field. If one is not, terminate with ``no'',
and if all are, use Algorithm \ref{CMtestalgor} to compute $\bar{\,\,}$ 
on each $A_\Q/\m$ and thus on $A_\Q \isom \prod_\m A_\Q/\m$.
\item
Express the given $\Z$-basis for $A$ with respect to the $\Q$-basis
for $A_\Q$ obtained in Step (ii).
\item
Compute the matrix for $\bar{\,\,}$ with respect to the $\Z$-basis for $A$.
If all entries are integers, then output ``yes'' and this matrix, and otherwise
terminate with ``no''.
\end{enumerate}
\end{algorithm}

\begin{prop}
\label{CMorddetalgpf}
Algorithm \ref{CMorddetalg} is correct and runs in  polynomial time.
\end{prop}

\begin{proof}
The algorithm is correct by Remarks \ref{CMalgrems}, \ref{CMalgalgorrem},
and \ref{CMordiffresult} 
(since $\Delta_{A/\Z} \neq 0$ if and only if every  element of $A$
is separable),
and runs in polynomial time since each step does. 
\end{proof}

\section{$A$-lattices}
\label{Alatticesect}

Throughout this section $A$ is a CM-order, except for Lemma \ref{Linvremarks}.
Suppose that $L$ is an $A$-module.
Then there is an $A$-module $\overline{L}$ with a group isomorphism
$\bar{\,\,} : L \to \overline{L}$  
that is semi-linear,
i.e., $\overline{rx} = \bar{r}\cdot\bar{x}$ for all $r\in A$
and $x\in L$.
The module $\overline{L}$ is easy to construct. If $L=A$, one can take
$\overline{L}=A$, and take $\bar{\,}$ on $L$ to be the
same as $\bar{\,}$ on $A$.

Recall that we define an $A$-lattice $L$ to be a lattice that is
given an $A$-module structure with the property that for all
$a\in A$ and $x,y\in L$ one has
$\langle ax,y \rangle = \langle x,\bar{a}y \rangle$.
Recall the definition of $\hat{A}$ in Definition \ref{Ahatetcdef}.

\begin{prop}
\label{latticeequivdefs1}
Suppose 
$L$ is an $A$-lattice.
Then:
\begin{enumerate}
\item
if $x,y\in L$, then there exists a unique $z_{x,y}\in \hat{A}$
such that 
$$
\Tr_{A_\Q/\Q}(az_{x,y}) = \langle ax,y \rangle
$$
for all $a\in A$; 
\item
 there is a unique $A$-linear homomorphism
$\varphi = \varphi_L : L\otimes_A \bar{L} \to \hat{A}$ such that
$$
\Tr_{A_\Q/\Q}(\varphi(x\otimes\bar{y})) = \langle x,y \rangle  
$$
for all $x,y\in L$;
for this map $\varphi$ we have
\begin{enumerate}
\item
$\varphi(x\otimes \bar{y}) = z_{x,y}$ for all $x,y\in L$,
\item
$\varphi(x\otimes \bar{y}) = \overline{\varphi(y\otimes \bar{x})}$ for all $x,y\in L$,   
\item
$\varphi(x\otimes \bar{x})\in \hat{A}^+_{>0}$ for all $0\neq x\in L$. 
\end{enumerate}
\end{enumerate}
\end{prop}

\begin{proof}
Since
$
g : \hat{A} \isom \Hom_\Z(A,\Z)$, 
$b \mapsto (a \mapsto \Tr_{A_\Q/\Q}(ab))
$
is an isomorphism, for every 
$x,y\in L$ there exists a unique $z_{x,y}\in \hat{A}$
such that $g(z_{x,y})$ is the map $a \mapsto \langle ax,y \rangle$.
This proves (i).

It is straightforward to check that
the map $L \times \bar{L} \to \hat{A},
(x,\bar{y}) \mapsto z_{x,y}$ is $A$-bilinear.
Thus there exists a unique $A$-linear map
$\varphi : L \otimes_A \bar{L} \to \hat{A}, x\otimes\bar{y} \mapsto z_{x,y}$,
and by (i) we have
$\Tr_{A_\Q/\Q}(a\varphi(x\otimes\bar{y})) = \langle ax,y \rangle$
for all $x,y\in L$ and $a\in A$.

If a map $\varphi : L\otimes_A \bar{L} \to \hat{A}$ is $A$-linear and satisfies
$
\Tr_{A_\Q/\Q}(\varphi(x\otimes\bar{y})) = \langle x,y \rangle  
$
for all $x,y\in L$, then 
$\Tr_{A_\Q/\Q}(a\varphi(x\otimes\bar{y})) = \langle ax,y \rangle$
for all $x,y\in L$ and $a\in A$, so $\varphi(x\otimes\bar{y}) = z_{x,y}$ by (i),
giving the uniqueness in (ii).

Since for all $a\in A$ we have
$$
\Tr(az_{x,y}) = \langle ax,y \rangle = \langle x,\bar{a}y \rangle  = 
\langle \bar{a}y,x \rangle = \Tr(\bar{a}z_{y,x}) = \Tr(a\overline{z_{y,x}})
$$
it follows that $z_{x,y} = \overline{z_{y,x}}$ and thus
$\varphi(x\otimes \bar{y}) = \overline{\varphi(y\otimes \bar{x})}$
for all $x,y\in L$.

Substituting $x$ for $y$, it follows that 
$\varphi(x\otimes \bar{x}) \in \hat{A}^+$.
If $x\neq 0$ then 
$\langle x,x \rangle \neq 0$, so
$\Tr(\varphi(x\otimes \bar{x})) \neq 0$, so
$\varphi(x\otimes \bar{x}) \neq 0$.
Extending $\varphi$ $\R$-linearly, we have
$$
\Tr_{A_\R/\R}(a\bar{a}\varphi(x\otimes \bar{x})) = 
\langle a\bar{a}x,x \rangle = 
\langle \bar{a}x,\bar{a}x \rangle \ge 0
$$
for all $x\in L_\R$ and $a\in A_\R$.
The proof of Lemma 7.3(vii) of \cite{LwS} with $A_\R$ in the
role of $\R\langle G \rangle$ and $z=\varphi(x\otimes \bar{x})$ now
gives that
$\psi(\varphi(x\otimes \bar{x})) \ge 0$ for all
$\psi\in\Rhom_\R(A_\R,\C)$ and all $x\in L_\R$.
It follows now that 
$\varphi(x\otimes \bar{x})\in \hat{A}^+_{>0}$ for all $0\neq x\in L$,
and we have (ii).
\end{proof}

\begin{prop}
\label{latticeequivdefs2}
Suppose 
$L$ is a finitely generated $A$-module,
and $\varphi = \varphi_L : L\otimes_A \bar{L} \to \hat{A}$ 
is an $A$-linear homomorphism such that
\begin{enumerate}
\item
$\varphi(x\otimes \bar{y}) = \overline{\varphi(y\otimes \bar{x})}$ for all $x,y\in L$,  and
\item
$\varphi(x\otimes \bar{x})\in \hat{A}^+_{>0}$ for all $0\neq x\in L$. 
\end{enumerate}
Then $L$ is an $A$-lattice with respect to the inner product
$$
\langle x,y \rangle = \Tr_{A_\Q/\Q}(\varphi(x\otimes\bar{y})).
$$
\end{prop}

\begin{proof}
Define 
$
\langle  \,\, , \,\, \rangle : L \otimes_A \bar{L} \to \Z$ by
$\langle x,y \rangle = \Tr_{A_\Q/\Q}(\varphi(x\otimes \bar{y})).
$
Note that the image lies in $\Z$ by the definition of $\hat{A}$,
and $\Z$-bilinearity is also clear.
We have
$$
\langle x,y \rangle = \Tr(\varphi(x\otimes \bar{y})) =
\Tr(\overline{\varphi(y\otimes \bar{x})}) =
\Tr({\varphi(y\otimes \bar{x})}) =
\langle y,x \rangle.
$$
If $x\neq 0$ then
$$
\langle x,x \rangle = \Tr_{A_\Q/\Q}(\varphi(x\otimes \bar{x}))
 = \sum_{\psi\in\Rhom(A_\Q,\C)}\psi(\varphi(x\otimes \bar{x})) > 0,
$$
the  inequality holding since each 
$\psi(\varphi(x\otimes \bar{x}))$ is real and non-negative, 
and at least one is positive. 
By the $A$-linearity of $\varphi$ we have  
$$
\langle ax,y \rangle = a\Tr_{A_\Q/\Q}(\varphi(x\otimes \bar{y}))
= \langle x,\bar{a}y \rangle.
$$
\end{proof}

\begin{defn}
\label{Alatticeinvertdef}
An $A$-lattice $L$ is {\em invertible} if the values of 
the map $\varphi_L$ of Proposition \ref{latticeequivdefs1}
all lie in $A$ and
the map $\varphi_L : L\otimes_A \bar{L} \to {A}$ is an isomorphism
of $A$-modules.
\end{defn}

\begin{rems}
\begin{enumerate}[leftmargin=*]
\item
For the  standard $A$-lattice $L=A$ we have
$\varphi_A(x\otimes \bar{y}) =x\bar{y}$ and 
$\langle x,y \rangle = \Tr_{A/\Z}(x\bar{y})$. 
The standard $A$-lattice is invertible since the map $A\otimes_A \bar{A} \to A$,
$x\otimes\bar{y}\mapsto x\bar{y}$ is an isomorphism.
\item
Invertibility is preserved under $A$-lattice isomorphisms. 
\end{enumerate}
\end{rems}

\begin{defn}
An $A$-module $L$ is {\em invertible}
if there is an  $A$-module $M$ such that
$L \otimes_{A} M$ and $A$
are isomorphic as $A$-modules.
\end{defn}

\begin{rem}
If $L$ is an invertible $A$-lattice, then $L$ is an invertible $A$-module.
\end{rem}

\begin{lem}
\label{Linvremarks}
If $A$ is a reduced order and $L$ is an invertible $A$-module, then 
$L_\Q$ and $A_\Q$ are isomorphic as $A_\Q$-modules, and 
$\rank_\Z(L) = \rank_\Z(A)$.
\end{lem}

\begin{proof}
We use the argument that shows (c) $\Rightarrow$ (a)
of Theorem 11.1 in \cite{LwS}.
Since $A_\Q$ is a product of finitely many fields $A_\Q/\m$ with $\m\in\Maxspec(A)$,
and $L_\Q$ is an $A_\Q$-module,
we have $L_\Q = \prod_\m V_\m$ where $V_\m$ is a vector space over
$A_\Q/\m$. Let $d_\m(L) = \dim(V_\m)$. 
Since $L$ is invertible, there is an $A$-module $M$ such that
$L_\Q \otimes_{A_\Q} M_\Q \cong A_\Q$.
Thus, $d_\m(L)d_\m(M) = d_\m(A) = 1$, so $d_\m(L) = 1 = d_\m(M)$.
The desired result now follows.
\end{proof}

\begin{notation}
\label{xynotat}
If $x,y\in L$, when we write $x\cdot \bar{y}$ or $x\bar{y}$ we mean
$\varphi(x\otimes \bar{y})$.
\end{notation}

\begin{rem}
If $L$ is an $A$-lattice, $x\in L$, and $x\bar{x}=1$, then
$\langle x,x \rangle = \rk_\Z(A)$,
by 
Propositions \ref{1equiv} and \ref{latticeequivdefs1}.
\end{rem}

We call a commutative ring $R$  {\em connected} if it has exactly two
idempotents.
The following result allows us to reduce our main algorithm
(Theorem \ref{mainthm}) to the case where $A$ is connected.

\begin{lem}
\label{decomposeLA}
Suppose 
$\I$ is the set of primitive idempotents of $A$.
Then:
\begin{enumerate}
\item
$A = \prod_{e\in\I} eA$
and each 
$eA$ is a CM-order (viewing $eA$ as a ring with identity $e$), 
\item
if $L$ is an $A$-lattice, then $L$ is the orthogonal sum
$\bigperp_{e\in\I} eL$ and each $eL$ is an $eA$-lattice,
\item
if $L$ is an invertible $A$-lattice, then 
each $eL$ is an invertible $eA$-lattice.
\end{enumerate}
\end{lem}

\begin{proof}
Since $\I$ is the set of primitive idempotents of $A$ we have 
$1 = \sum_{e\in\I} e$, so $A = \prod_{e\in\I} eA$ and $L = \bigoplus_{e\in\I} eL$.
Suppose $e\in\I$. Then $\psi(e)\in\{ 0,1\}$ for all $\psi\in\Rhom(A,\C)$,
so $\psi(e) = \overline{\psi(e)} = \psi(\overline{e})$
for all $\psi$. Thus, $e=\bar{e}$, so $\overline{eA} = \bar{e}\bar{A} = eA$.
Parts (i) and (ii) now follow easily from 
Definition \ref{CMorddef} and 
the definition of an $A$-lattice. 
Part (iii) follows from the definition of invertibility since 
$1\otimes 1 = \sum_{e\in\I}(e\otimes \bar{e})$
and 
$(e\otimes \bar{e})(L \otimes_A \bar{L}) = eL\otimes_{eA} \overline{eL}$.
\end{proof}

\section{Reduced bases}
\label{redbasautsect}

The main result of this section is Proposition \ref{LLLlem}.
It shows that there exists $B\in\R$ depending only on 
the CM-order $A$,
and polynomially bounded in the length of the data specifying $A$, such that
for each invertible $A$-lattice $L$, 
the length of the data specifying $L$ is bounded by $B$.
It is an analogue of Proposition 3.4 of \cite{LwS} 
(see also Lemma 3.12 of \cite{LenSil}), which was for
integral unimodular lattices.
It allows us to show that the Witt-Picard group of $A$ is finite
(Theorem \ref{WPgpfin} below),
and helps to show, as in \cite{LwS}, that the algorithms 
associated with Theorem \ref{LMmultalg} run in polynomial time.

\begin{defn}
\label{LLLreduceddefn}
If $\{ b_1,\ldots,b_m\}$ is a basis for a lattice $L$,
and $\{ b_1^\ast,\ldots,b_m^\ast\}$
is its Gram-Schmidt orthogonalization, 
and $b_i = b_i^\ast + \sum_{j=1}^{i-1} \mu_{ij}b_j^\ast$
with $\mu_{ij}\in \R$,
then
$\{ b_1,\ldots,b_m\}$ is {\bf LLL-reduced} if 
\begin{enumerate}
\item
$|\mu_{ij}| \le \frac{1}{2}$ for all $j<i\le m$, and 
\item
$|b_i^\ast|^2 \le 2|b_{i+1}^\ast|^2$
for all $i<m$.
\end{enumerate}
\end{defn}

The LLL basis reduction algorithm \cite{LLL} takes as input a lattice,
and produces an LLL-reduced basis of the lattice, in polynomial time.

Recall the definition of the inner product $( \, \, \, , \,\, \,  )$ 
in Definition \ref{innerproddef}.

\begin{lem}
\label{axlem}
If $A$ is a CM-order, $L$ is an $A$-lattice, $a\in A$, and $x\in L$, then 
$\langle ax,ax\rangle \le (a,a)\langle x,x\rangle$.
\end{lem}

\begin{proof}
If $\sigma\in\Rhom(A_\Q,\C)$, then
$\sigma(a\bar{a}) = \sigma(a)\overline{\sigma({a})} \in\R_{\ge 0}$, 
and $\sigma(\varphi(x\otimes\bar{x}))  \in\R_{\ge 0}$ 
by Proposition \ref{latticeequivdefs1}(ii)(c).
Then by Proposition \ref{latticeequivdefs1}(ii) we have
\begin{multline*}
\langle ax,ax\rangle =
\Tr_{A_\Q/\Q}(\varphi(ax\otimes\overline{ax})) =
\Tr_{A_\Q/\Q}(a\bar{a}\varphi(x\otimes\overline{x})) \\ =
\sum_{\sigma\in\Rhom(A_\Q,\C)}\sigma(a\bar{a})\sigma(\varphi(x\otimes\bar{x})) \le
\left(\sum_{\sigma}\sigma(a\bar{a})\right)
\left(\sum_{\sigma}\sigma(\varphi(x\otimes\bar{x}))\right) \\ =
(a,a)\Tr_{A_\Q/\Q}(\varphi(x\otimes\overline{x})) =
(a,a)\langle x,x\rangle.
\end{multline*}
\end{proof}

\begin{defn}
\label{detlatdef}
We define the determinant $\det(L)$ of a lattice $L$ to be the determinant of its Gram matrix, or
equivalently, the order of the cokernel of the map $L \to \Hom(L,\Z)$,
$x \mapsto (y \mapsto \langle x,y \rangle)$.
\end{defn}

\begin{lem}
\label{invDelta}
If $L$ is an invertible $A$-lattice,  then
$\det(L) = \det(A) = |\Delta_{A/\Z}|$.
\end{lem}

\begin{proof}
Consider the maps:
$$
L \to \Hom_{A}(\overline{L},A)
\to \Hom(\overline{L},\Z)
\to \Hom({L},\Z)
$$
where the left-hand map is the $A$-module isomorphism
$x \mapsto (\bar{y} \mapsto \varphi(x \otimes \overline{y}))$
with inverse
$f \mapsto (\id_L\otimes f)\circ \varphi^{-1}(1)$,
the middle map is $f\mapsto \Tr_{A/\Z}\circ f$, 
and  the right-hand map is the group isomorphism
$g \mapsto (y\mapsto g(\bar{y}))$.
By Proposition \ref{latticeequivdefs1}, 
the composition is the map 
$x \mapsto (y \mapsto  \langle x,y\rangle)$ of Definition \ref{detlatdef}.
We will show that the cokernel of the middle map
has order $|\Delta_{A/\Z}|$.
By the definition of $\Delta_{A/\Z}$, this holds with $A$ in place of $\overline{L}$,
and we next reduce to that case.
Since $L$ is invertible, we may identify $\overline{L}_\Q$ with $A_\Q$
by Lemma \ref{Linvremarks}.
Multiplying $\overline{L}$ by a sufficiently large positive integer, we may assume that
$\overline{L} \subset A$.
Let 
$$
L' = \{ a\in A_\Q : a\overline{L} \subset A\}.
$$
Consider the commutative diagram 
$$
\xymatrix@C=40pt{
~L' = \Hom_A(\overline{L},A)~ \ar[r] & \Hom(\overline{L},\Z) ~ \\
A = \Hom_A(A,A)~ \ar[r]\ar[u] & ~\Hom(A,\Z)~ \ar[u]
}
$$
where the vertical maps are the restriction maps.
The orders of the cokernels of the left and right maps are,
respectively, $(L':A)$ and $(A:\overline{L})$.
It suffices to show that these two numbers are equal.

We have $A \to L \otimes_A \overline{L} \onto L \cdot \overline{L}$
where the first map is the inverse of the isomorphism $\varphi_L$, so
$L \cdot \overline{L}$ is a principal ideal of $A$.
Hence  $I=\overline{L}$ is an invertible $A$-ideal
of finite index, 
and 
$I^{-1} = \{ a\in A_\Q : aI \subset A\} =L'$.
It remains to show that $(I^{-1}:A) = (A:I)$.
The map $J \mapsto J\cdot I$ from the set of intermediate $A$-modules
of $I^{-1} \supset A$ to the set of intermediate $A$-modules
of $A \supset I$ is a bijection with inverse $K \mapsto K \cdot I^{-1}$.
So a composition chain of $I^{-1}/A$ gives
a composition chain of $A/I$. Thus it suffices to prove
that if $J/J'$ is simple then $J/J' \cong J \cdot I/J' \cdot I$.
If $J/J' \cong A/\m$ with $\m\in\Maxspec(A)$, then
$J \cdot I/J' \cdot I$ is also simple and
annihilated by $\m$, so is also isomorphic to $A/\m$.
This gives the desired result. 
\end{proof}

We specify an $A$-lattice $L$ by giving $A$ as before,
$m=\rk_\Z(L)$,
the Gram matrix $(\langle b_i,b_j \rangle)_{i,j=1}^m$ with respect
to a $\Z$-basis $\{b_1,\ldots,b_m\}$ for $L$,
and $d_{ijk}\in\Z$ 
for $i \in\{1,\ldots, n\}$ and
$j,k \in\{1,\ldots, m\}$ such that
$\alpha_i b_j = \sum_{k=1}^m d_{ijk} b_k$ for all $i$ and $j$,
with respect to the same $\Z$-basis $\{ \alpha_i\}_{i=1}^n$ that
was used for the system of integers $\{ b_{ijk}\}_{i,j,k=1}^n$ used to specify $A$.
We always work with LLL-reduced bases for $A$-lattices,
as we explained for $A$ at the beginning of Section \ref{CMorderssect}.

If $x \in L_\R$ let $|x| = \langle x,x \rangle^{1/2}$, and
if $a \in A_\R$ let $|a| = ( a,a )^{1/2}$.

\begin{prop}
\label{LLLlem}
If $A$ is a CM-order, 
$n$ is its rank, 
$L$ is an invertible $A$-lattice,  
$\{ b_1,\ldots,b_m\}$ is an LLL-reduced basis for 
$L$, and $\{ b_1^\ast,\ldots,b_m^\ast\}$
is its Gram-Schmidt orthogonalization, then $m=n$ 
and:
\begin{enumerate}
\item
$2^{1-i} \le |b_i^\ast|^2 \le 2^{n-i}|\Delta_{A/\Z}|$ for all $i$, 
\item
$|b_i|^2 \le 2^{n-1}|\Delta_{A/\Z}|$  for all $i$,
\item
$|\langle b_i,b_j\rangle| \le 2^{n-1}|\Delta_{A/\Z}|$ for all $i$ and $j$,
\item
$|d_{ijk}|, |b_{ijk}| \le (3\sqrt{2})^{n-1}|\Delta_{A/\Z}|$
 for all $i$,  $j$, and $k$.
\end{enumerate}
\end{prop}

\begin{proof}
The proof generalizes our proof of Proposition 3.4 of \cite{LwS}
(and corrects some typographical errors therein). 
Since $L$ is an  invertible $A$-lattice, we have $m=n$ and
$\det(L) = |\Delta_{A/\Z}|$, by Lemma \ref{invDelta}.
It follows from Definition \ref{LLLreduceddefn}(ii) that for all 
$1 \le i \le j \le n$ we have
$|b_i^\ast|^2 \le 2^{j-i}|b_j^\ast|^2$, so
for all $i$ we have 
$
2^{1-i}|b_1^\ast|^2 \le |b_i^\ast|^2 \le 2^{n-i}|b_n^\ast|^2.
$
Since $L$ is integral we have 
$
|b_1^\ast|^2 = |b_1|^2 = \langle b_1,b_1\rangle \ge 1,$ 
so $|b_i^\ast|^2 \ge 2^{1-i}$.
Letting $L_i = \sum_{j=1}^i \Z b_j$, we have 
$
|b_i^\ast|^2 = \det(L_i)/\det(L_{i-1}).
$
Since $L$ is integral 
we have
$|b_n^\ast|^2 = \det(L_n)/\det(L_{n-1}) \le |\Delta_{A/\Z}|,$
so $|b_i^\ast|^2 \le 2^{n-i}|\Delta_{A/\Z}|$, 
giving (i).

Following the proof of Proposition 3.4(ii,iii) of \cite{LwS}
now gives (ii) and (iii).

Define $\{ c_1,\ldots,c_n\}$ to be the $\Q$-basis of $L_\Q$ that is dual to
$\{ b_1,\ldots,b_n\}$, i.e., $\langle c_i,b_j\rangle = \delta_{ij}$
for all $i$ and $j$, where $\delta_{ij}$ is the Kronecker delta symbol.
Then $d_{ijk} = \langle c_j,\alpha_i b_j\rangle$,
so
$$
|d_{ijk}| \le |c_j||\alpha_i b_j| \le |c_j||\alpha_i||b_j| \le 2^{n-1}|\Delta_{A/\Z}||c_j| 
$$
by the Cauchy-Schwarz inequality, 
Lemma \ref{axlem}, and (ii)
applied to the $A$-lattices $L$ and $A$.
The proof of Proposition 3.4(iv) of \cite{LwS} shows that
$|c_j|^2  \le (9/2)^{n-1}$, 
and this gives the desired bound on $|d_{ijk}|$ in (iv).
Applying this to the standard $A$-lattice $A$ 
(recall that $\{ \alpha_i \}_{i=1}^n$ is LLL-reduced) gives
the desired bound on $|b_{ijk}|$.
\end{proof}

\section{Short vectors in lattice cosets}
\label{latticecosetssect}
We show how to find the unique ``short'' vector in a
suitable lattice coset, when such a vector exists.

\begin{defn}
\label{shortdef}
Suppose $A$ is a CM-order and $L$ is an $A$-lattice.
We say $x\in L$ is {\em short} if
$\varphi(x\otimes \bar{x}) =1$, 
where $\varphi$ is the map from Proposition \ref{latticeequivdefs1}.
\end{defn}

Shortness is preserved by $A$-lattice isomorphisms.
Recalling Notation \ref{xynotat}, the element $x$ is short if
and only if $x\bar{x} =1$.
Hence $\langle x,x\rangle = \rk_\Z(A)$ when $x$ is short.

The following algorithm is 
an analogue of Algorithm 4.2 of \cite{LwS}.
We will use it in Algorithms \ref{ingredient4} and \ref{almostmainalgor} below.

\begin{algorithm} 
\label{findingekmalg2}
Given a CM-order $A$, an $A$-lattice $L$ of $\Z$-rank $n$, an $A$-ideal $\mathfrak{a}$
of finite index in $A$ such that
\begin{equation}
\label{betabetage}
\langle \beta,\beta\rangle \ge (2^{n/2}+1)^2\rk_\Z(A)
\text{ for all $\beta\in \a L \smallsetminus \{ 0 \}$},
\end{equation} 
and $C\in L/\mathfrak{a}L$, the algorithm computes
all $y\in C$ with $\langle y,y\rangle=\rank_\Z(A)$.

Steps:
\begin{enumerate}
\item
Compute an LLL-reduced basis for $\mathfrak{a}L$ and use it
as in \S 10 of \cite{HWLMSRI} to 
compute $y\in C$ such that 
$
\langle y,y \rangle \le (2^n -1)\langle x,x \rangle
$
for all $x\in C$,
i.e., find an approximate solution to the  
shortest vector problem.
\item
Compute $\langle y,y \rangle$. 
\item
If $\langle y,y \rangle =\rank_\Z(A)$, output $y$.
\item 
If $\langle y,y \rangle \neq \rank_\Z(A)$, output
``there is no $y\in C$ with $\langle y,y\rangle=\rank_\Z(A)$''.
\end{enumerate}
\end{algorithm}

The following result is used to prove Proposition \ref{ingredient4pf}.

\begin{prop}
\label{findingekmprop2}
Algorithm \ref{findingekmalg2} is correct and runs in polynomial time.
Further, the number of $y$ output by the algorithm is $0$ or $1$,
and if such a $y$ exists then it is the unique shortest element of $C$.
\end{prop}

\begin{proof}
Let $y\in C$ be as computed in Step (i). Then
$\langle y,y \rangle \le (2^n -1)\langle x,x \rangle$
for all $x\in C$. Suppose
$z\in C$ 
is such that
$\langle z,z \rangle \leq \rank_\Z(A)$,
and let $\beta = z-y\in \mathfrak{a}L$. Then
$$
\langle \beta,\beta\rangle 
\le \left(\langle z,z \rangle^{1/2} + \sqrt{2^n-1}\langle z,z \rangle^{1/2}\right)^2 
< (2^{n/2}+1)^2\rank_\Z(A),
$$
so $\beta = 0$ by \eqref{betabetage} and $z=y$.
It follows that the algorithm finds all 
$y\in C$ with $\langle y,y \rangle=\rank_\Z(A)$,
there is at most one such, and 
if one exists then it is the unique shortest element of $C$.
\end{proof}

\begin{rem}
Note that $2^{2(n+1)} \ge (2^{n/2}+1)^2n$.
Thus if $L$ is an $A$-lattice, $n = \rank_\Z(A) = \rank_\Z(L)$, and 
$\a = 2^{n+1}A$, then 
\eqref{betabetage} holds.
We will make special use of the ideal $2^{n+1}A$ in
Algorithms \ref{almostmainalgor} and \ref{ingredient4}.
\end{rem}

\section{Short vectors and regular elements}
\label{regsect}

\begin{defn}
\label{regulardef}
Suppose $A$ is a commutative ring and $L$ is an  
$A$-module.
An element $x\in L$ is {\em regular} (or regular in $L$) if the map $A \to L$
defined by $a \mapsto a x$ is injective.
\end{defn}

Recall (Notation \ref{xynotat}) that 
$x\bar{y}$ is shorthand for $\varphi(x\otimes \bar{y})$.

\begin{prop}
\label{regularequiv}
Suppose $A$ is a CM-order, $L$ is an $A$-lattice, and $x\in L$.
Then the following are equivalent:
\begin{enumerate}
\item
$x$ is regular,
\item
$x\bar{x} \in \hat{A}^+_{\gg 0}$,
\item
$x\bar{x}$ is regular (in $A_\Q$). 
\end{enumerate}
\end{prop}

\begin{proof}
Let $(y_\r)_{\r\in\Minspec(A)}$ denote the image 
of $y\in A_\Q$ under the natural isomorphism
$A_\Q \isom \prod_{\r\in\Minspec(A)} A_\r$,
where each $A_\r$ is a field (cf.\ Remark \ref{CMalgrems}).
Then $y\in A_\Q$ is regular in $A_\Q$ if and only if $y_\r \neq 0$ for all $\r$.
This implies that (ii) and (iii) are equivalent,
by Proposition \ref{latticeequivdefs1}(ii)(c).

Suppose $x$ is regular. If $0\neq a\in A$, then $ax \neq 0$, so 
\begin{equation}
\label{trgr0eqn}
0 \neq \langle ax,ax \rangle  = \Tr(a\bar{a}(x\bar{x})).
\end{equation}
If $\r\in\Minspec(A)$ and $(x\bar{x})_\r = 0$ in $A_\r$, then
there exists $b\in A_\Q\smallsetminus \{ 0\}$ such that
$b(x\bar{x}) =0$,
so there exists $a\in A\smallsetminus \{ 0\}$ such that
$a(x\bar{x}) =0$.
Thus, $a\bar{a}(x\bar{x}) =0$,
so $\Tr(a\bar{a}(x\bar{x})) =0$,
contradicting \eqref{trgr0eqn}.
It follows that (i) implies (ii).

Next we show that (ii) implies (i). 
Suppose $a\in A$ and $ax=0$. Then
$a(x\bar{x}) = (ax)\bar{x} = 0$.
By (ii) we have 
$x\bar{x} \in \hat{A}^+_{\gg 0} \subset A_\Q^\ast$.
Thus $a=0$, giving (i).
\end{proof}

Recall the definition of {\em short} in Definition \ref{shortdef}.

\begin{prop}
\label{shortequiv}
Suppose $A$ is a CM-order, $L$ is an $A$-lattice, 
$\varphi(L\otimes \bar{L}) \subset A$,
and $x\in L$.
Then the following are equivalent:
\begin{enumerate}
\item
$x$ is short,
\item
$x$ is regular and $\langle x,x \rangle =\rank_\Z(A)$.
\end{enumerate}
\end{prop}

\begin{proof}
That (i) implies (ii) follows from
Proposition \ref{regularequiv} and $\Tr(1) =\rank_\Z(A)$.

Conversely, assume (ii) and
let $a = x\bar{x}$.  
Then $a\in {A}^+_{\gg 0}$ by Proposition \ref{regularequiv}, and 
$\Tr(a) = \langle x,x \rangle =\rank_\Z(A)$, so 
by Proposition \ref{1equiv} we have $a=1$.
\end{proof}

The next result may be viewed as a variation on Kronecker's theorem
that every algebraic integer all of whose conjugates lie on the
unit circle must be a root of unity.
We will use it to prove Theorem \ref{shortthm}(iv).

\begin{prop}
\label{muAequiv}
Suppose $A$ is a CM-order
and $a\in A$.
Then the following are equivalent:
\begin{enumerate}
\item
$a\in\mu(A)$,
\item
$a$ is regular and $\Tr(a\bar{a}) = ( a,a ) =\rank_\Z(A)$, 
\item
$a$ is regular and $\Tr(a\bar{a}) = ( a,a ) \le \rank_\Z(A)$, 
\item
$a\bar{a}=1$.
\end{enumerate}
\end{prop}

\begin{proof}
That (i) $\Rightarrow$ (ii) follows by applying Proposition \ref{shortequiv} 
to the standard $A$-lattice $L=A$.
The implication 
(ii) $\Rightarrow$ (iii) 
is clear.
For (iii) $\Rightarrow$ (iv), suppose we have (iii).
Then $\bar{a}$ is regular, so $a\bar{a}$ is regular.
By Proposition \ref{regularequiv},
$a\bar{a} \in A^+_{\gg 0}$.
Since $\Tr(a\bar{a}) \le \rank_\Z(A)$, by Proposition \ref{1equiv}
we have $a\bar{a}=1$ as desired.

To show  (iv) $\Rightarrow$ (i), 
suppose $a\bar{a}=1$.
We have $A_\Q \cong \prod_{\r\in\Minspec(A)}A_\r$
with each localization $A_\r$ being a number field, and
the components $a_\r$ of $a$
are algebraic integers all of whose conjugates lie on
the unit circle, so each $a_\r$ is a root of unity. Thus, $a\in\mu(A)$. 
\end{proof}

\begin{ex}
For an example of a CM-order with a vector shorter than a
``short'' one, 
suppose that $A_1$ and $A_2$ are non-zero CM-orders and let $A = A_1\times A_2$,
a disconnected order.
Then the unit element $1\in A$ satisfies 
$\langle 1,1\rangle = \Tr(1\cdot\bar{1}) = \Tr(1) = \rank_\Z(A)$,
so by Proposition \ref{shortequiv} with $L=A$ the vector $1$ is ``short''.
For $(1,0)\in A_1\times A_2 = A$ we have
$$
\langle (1,0),(1,0)\rangle = \Tr((1,0)\cdot\overline{(1,0)}) =  \rank_\Z(A_1)
< \rank_\Z(A)
$$
and similarly
$$
\langle (0,1),(0,1)\rangle  =  \rank_\Z(A_2) < \rank_\Z(A),
$$
giving shorter vectors than our ``short'' vector $1=(1,1) \in A$.
\end{ex}

\begin{ex}
For an example of a {\em connected} order with a non-zero vector shorter than a
``short'' one, let
$$
A = \{ (x_1,x_2,x_3,x_4,x_5)\in\Z^5 : \text{ all $x_i$ have the same parity}\}
$$
with coordinate-wise multiplication.
Then $A$ is a subring of $\Z^5$ of index $16$, and $A$ is a connected order.
The element $x=(2,0,0,0,0)\in A$ has 
$\langle x,x\rangle = 4$,
while $\rank_\Z(A) = \langle 1,1\rangle = 5 > 4$.
\end{ex}

\begin{ex}
Let
$$
A = \{ (x_1,x_2,x_3,x_4)\in\Z^4 : \text{ all $x_i$ have the same parity}\}
$$
with coordinate-wise multiplication.
The element $x=(2,0,0,0)\in A$ has 
$\langle x,x\rangle = 4 =\langle 1,1\rangle$.
While $1$ is regular, $x$ is not (since in $A$,
an element is regular if and only if no coordinate is $0$).
\end{ex}

\section{Vigilant sets and lower bounds}
\label{lowerbdssect}

Suppose $A$ is a CM-order.
The main result of this section is Proposition \ref{xxlowerbd},
which for any $A$-ideal $\a$ that can be written as a product of
finitely many maximal ideals,
finds a lower bound for 
$\min\{\langle \beta,\beta \rangle : \beta\in \mathfrak{a}L \smallsetminus \{ 0\}\}$
in terms of $\mathfrak{a}$, valid for all $A$-lattices $L$
for which the image of $\varphi$ is contained in $A$.
We will use it to prove Proposition \ref{Usetalgorpf}.
We start with some lemmas.

See Corollary 2.5 of \cite{AtiyahMcD} for the following version
of Nakayama's Lemma.

\begin{prop}[Nakayama's Lemma]
\label{Nakayama}
Suppose $A$ is a commutative ring, $L$ is a finitely generated
$A$-module, and $\a$ is an ideal of $A$ such that $\a L = L$.
Then there exists $x\in 1 + \a \subset A$ such that $xL = 0$.
\end{prop}

\begin{lem}
\label{Fact1lem}
Suppose $A$ is an order, $I \subset A_\Q$ is a fractional 
$A$-ideal, and $\a \subsetneq A$ is an ideal.
Then $\a I \subsetneq I$.
\end{lem}

\begin{proof}
If not, then $\a I = I$, so  Nakayama's Lemma 
(Proposition \ref{Nakayama}) gives 
$x\in 1 + \a \subset A$ such that $xI=0$.
Then $xI_\Q = 0$. 
But $I_\Q = A_\Q$. So $x = x\cdot 1 \in x\cdot A_\Q = xI_\Q = \{ 0\}$.
Thus, $1\in\a$, so $\a =A$, a contradiction.
\end{proof}

Recall that if $\p\in\Spec(A)$, then $\N(\p) = \#(A/\p)$.

\begin{lem}
\label{bcnormprop}
Suppose $A$ is an order, 
$\p_1,\ldots,\p_m \in\Maxspec(A)$,  
and
$\cc$ is a fractional $A$-ideal.  
Then
$$
\#(\cc/\p_1\cdots\p_m\cc) \ge \prod_{i=1}^m\Norm(\p_i).
$$
\end{lem}

\begin{proof}
We proceed by induction on $m$. The case $m=0$ is clear.
For $m>0$, letting $J$ denote the fractional $A$-ideal
$\p_1\cdots\p_{m-1}\cc$ we have
$$
\#(\cc/\p_1\cdots\p_m\cc) = \#(\cc/J)\#(J/\p_mJ).
$$
By Lemma \ref{Fact1lem} we have $J \neq \p_mJ$.
Thus,
$\dim_{A/\p_m}(J/\p_mJ) \ge 1$,
so $\#(J/\p_mJ) \ge \Norm(\p_m)$.
\end{proof}

\begin{defn}
\label{vigilantdef}
Suppose $A$ is a reduced order. 
We will say that a set $S$ of maximal ideals of $A$ is a {\em vigilant} set for $A$ if
for all $\r\in\Minspec(A)$ there exists $\L\in S$ such that $\r \subset \L$.
\end{defn}
Being a vigilant set for $A$ is equivalent to the natural map $A \to \prod_{\L\in S}A_\L$
being injective. 
If $S \subset \Maxspec(A)$ 
and $\r\in \Minspec(A)$,
let 
$$
S(\r) = \{ \L\in S: \r \subset \L\}.
$$ 
Then $S = \bigcup_{\r\in \Minspec(A)} S(\r)$, and $S$ is 
a vigilant set for $A$ if and only if each
$S(\r)$ is non-empty.
If $S$ is vigilant, we think of
$S$ as ``seeing'' all the irreducible
components of $\Spec(A)$.

\begin{prop}
\label{xxlowerbd}
Suppose that $A$ is a CM-order,  $n$ is its rank, 
$L$ is an $A$-lattice such that
the map $\varphi$ of Proposition \ref{latticeequivdefs1}
takes values in $A$, and
$S$ is a finite subset of $\Maxspec(A)$. 
Suppose $t : S\to \Z_{\ge 0}$ is a function, 
and $\mathfrak{a} = \prod_{\p\in S}\p^{t(\p)}$.
For $\r\in \Minspec(A)$,  
let $\dd_\r  = \rank_\Z(A/\r)$.
Then: 
\begin{enumerate}
\item
for all non-zero $\beta \in \mathfrak{a}L$ we have
$$
\langle \beta,\beta \rangle \ge 
\min_{\r\in \Minspec(A)} {\dd_\r} 
\prod_{\p\in S(\r)}\Norm(\p)^{\frac{2t(\p)}{\dd_\r}};
$$
\item
if $S$ is vigilant and
$t(\p)\ge {n(n+1)}$ for all $\p\in S$, then
$
\langle \beta,\beta \rangle \ge 
(2^{n/2}+1)^2n
$
for all $\beta \in \mathfrak{a}L \smallsetminus \{ 0\}$.
\end{enumerate}
\end{prop}

\begin{proof}
Suppose $\r \in \Minspec(A)$.
Then $\r = \ker(A\to A_\r, \alpha\mapsto \alpha_\r),$
and $A_\r$ is a zero-dimensional 
local ring with no nilpotent elements,
so it is a field, namely 
the field of fractions of $A/\r$.
(Note that $A/\r$ is a domain but not a field.)
For $C\subset A$, let $C(\r)$ denote the image of $C$ in $A_\r$.
We have
$\mathfrak{a}(\r) = 
\prod_{\p\in S(\r)}\p(\r)^{t(\p)}$ 
and 
$\bar{\mathfrak{a}}(\r) = 
\prod_{\p\in S(\r)}\overline{\p(\r)}^{t(\p)}$. 
If $\p\in S(\r)$, then 
$$
A(\r)/\p(\r) \cong (A/\r)/(\p/\r) \cong A/\p,
$$
 so
\begin{equation}
\label{indexeq}
\#(A(\r)/\p(\r)) = \Norm(\p).
\end{equation}

For (i), put $w = \beta\bar{\beta}  
\in \mathfrak{a}\bar{\mathfrak{a}}{A}$. 
Then $0 \neq w\in {A}^+_{>0}$.
Choose $\ff \in \Minspec(A)$ such that
$w\notin \ff$ (which we can do since $\bigcap_{\r\in \Minspec(A)} \r = (0)$).
Then $A/\ff \cong A(\ff) \subset A_\ff$, and 
$0 \neq w(\ff)  \in \mathfrak{a}\bar{\mathfrak{a}}A(\ff)$.
Then
\begin{align*}
\langle \beta,\beta \rangle 
&= 
\Tr_{A_\Q/\Q}(w) \ge 
\Tr_{A_\ff/\Q}(w(\ff)) = 
\sum_{\sigma\in\Rhom(A_\ff,\C)} \sigma(w(\ff)) 
\\
& = \dd_\ff\cdot\frac{1}{\dd_\ff} \sum_{\sigma} \sigma(w(\ff)) \\
&\ge \dd_\ff\left(\prod_{\sigma} \sigma(w(\ff))\right)^{1/\dd_\ff} \quad\text{by the arithmetic-geometric mean inequality}
 \\  
&=  \dd_\ff |\N_{A(\ff)/\Z}(w(\ff))|^{1/\dd_\ff}
\quad = \quad \dd_\ff [\#(A(\ff)/w(\ff) A(\ff))]^{1/\dd_\ff}
 \\  
&
\ge
{\dd_\ff} [\#(A(\ff)/\mathfrak{a}\bar{\mathfrak{a}}A(\ff))]^{1/\dd_\ff} 
 \\  
 &\ge
{\dd_\ff} 
\prod_{\p\in S(\ff)}\Norm(\p)^{2t(\p)/\dd_\ff} \quad \text{by \eqref{indexeq}, 
Lemma \ref{bcnormprop},
and $\Norm(\overline{\p}) = \Norm(\p)$}
 \\  
 &\ge
\min_{\r\in \Minspec(A)}{\dd_\r} 
\prod_{\p\in S(\r)}\Norm(\p)^{2t(\p)/\dd_\r}, 
\end{align*} 
giving (i).

For (ii), since $S$ is vigilant each $S(\r)$ is non-empty.
Since $1\le d_\r \le n$,
by (i) we have
$$
\langle \beta,\beta \rangle \ge 
\min_{\r\in \Minspec(A)} {\dd_\r} 
\prod_{\p\in S(\r)}\Norm(\p)^{\frac{2n(n+1)}{\dd_\r}}
\ge 2^{2n+2}
\ge 
(2^{n/2}+1)^2n.
$$
\end{proof}

\begin{ex}
Let $A = \Z \times_{\F_3} \Z$. 
Then $\Spec(A)$ is connected, and
is the union of 2 copies of $\Spec(\Z)$ that are identified at the prime $3$.
The minimal prime ideals of $A$ are
$\r_1 = \{ 0\} \times 3\Z$ and $\r_2 = 3\Z \times \{ 0\}$.
Let $\p = (2\Z \times \Z) \cap A$ and
$S=\{ \p \}$.
Then  $S(\r_1) = S$, but
$S(\r_2)$ is empty  
so $S$ is not vigilant.
Let $L=A$ be the standard $A$-lattice.
For every $t\in\Z_{> 0}$, 
one has $\p^t = (2^t\Z \times \Z) \cap A$.
Hence, independently of $t$, one has
$\beta=(0,3)\in \p^t = \p^t L$, and
$\langle \beta,\beta\rangle = \trace((0,3)) = 9$.
Thus, the hypothesis that $S$ is vigilant cannot be removed 
in Proposition \ref{xxlowerbd}(ii).
\end{ex}

\section{Ideal lattices}
\label{ideallatsect}

The proof of Theorem 8.2 of \cite{LwS} carries over essentially
verbatim, with $\Z\langle G\rangle$ replaced by $A$ and
 $\Q\langle G\rangle$ replaced by $A_\Q$, to show:

\begin{thm}
\label{Iwmisclem}
Suppose $A$ is a CM-order, $I \subset A_\Q$ is a fractional $A$-ideal, 
and $w \in (A^+_\Q)_{\gg 0}$. Suppose 
that 
$
I\bar{I} \subset \hat{A} w.
$ 
Then:
\begin{enumerate}
\item
$\overline{w}=w$;
\item 
$w\in A_\Q^\ast$;
\item
$I$ is an $A$-lattice, with 
$
\varphi(x\otimes \bar{y}) = \frac{x\overline{y}}{w}
$
and 
$\langle x,y \rangle = \Tr_{A_\Q/\Q}\left(\frac{x\overline{y}}{w}\right)$.
\end{enumerate}
\end{thm}

\begin{notation}
\label{IwGlatdefn}
With $I$ and $w$ as in Theorem \ref{Iwmisclem},
define $L_{(I,w)}$ to be
the $A$-lattice $I$ with 
$\langle x,y \rangle = \Tr_{A_\Q/\Q}(x\bar{y}/w)$.
\end{notation}

The proof of Theorem 8.5 of \cite{LwS} carries over
(with $\Tr$ playing the role of the scaled trace function 
$t$ of \cite{LwS}) to give the following result, which
allows us to deduce Theorem \ref{mainIwthm} from Theorem \ref{mainthm}.

\begin{thm} 
\label{I1I2isom}
Suppose  $A$ is a CM-order, 
$I_1$ and $I_2$ are fractional $A$-ideals,
and $w_1, w_2 \in (A^+_\Q)_{\gg 0}$ 
satisfy
$I_1\overline{I_1} \subset \hat{A} w_1$ and 
$I_2\overline{I_2} \subset \hat{A} w_2$. 
Let $L_j = L_{(I_j,w_j)}$ for $j=1,2$.
Then sending $v$ to multiplication by $v$ gives a bijection from 
$$\{ v\in A_\Q :  I_1 = vI_2, w_1 = v\overline{v}w_2 \}
\quad \text{to} \quad  
\{ \text{$A$-isomorphisms $L_2 \isom L_1$} \}$$
and gives a bijection from
$$\{ v\in A_\Q : I_1 = vA, w_1=v\overline{v} \}
\quad \text{to} \quad  
\{ \text{$A$-isomorphisms 
$A \isom L_1$} \}.$$
In particular, 
$L_{1}$ is $A$-isomorphic to $A$ if and only if
there exists $v\in A_\Q$ such that $I_1 = (v)$ 
and $w_1=v\overline{v}$.
\end{thm}

\begin{rem}
If $I$, $w$, and $L_{(I,w)}$ are as in Theorem \ref{Iwmisclem}
and Notation \ref{IwGlatdefn}, then
$
\overline{L_{(I,w)}} = L_{(\overline{I},{w})}.
$ 
\end{rem}

\section{Invertible $A$-lattices}
\label{invAlatsect}

Recall the definition of invertible $A$-lattice from Definition \ref{Alatticeinvertdef}.
Theorem 11.1 of \cite{LwS} gave equivalent statements for
invertibility of a $G$-lattice.
The following example shows that the result does not 
fully extend to the case of  $A$-lattices,
while Theorem \ref{LwSthm} gives a part that does carry over.

\begin{ex}
We give an example of an $A$-lattice $L$
that is invertible as an $A$-module and satisfies $\det(L) = |\Delta_{A/\Z}|$,
but is not invertible as an $A$-lattice.
The CM-order $A = \Z\left[\frac{1+\sqrt{17}}{2},\sqrt{-1}\right]$
has 
$$
A^+ = \Z\left[\frac{1+\sqrt{17}}{2}\right], \quad
\hat{A} = \frac{1}{2\sqrt{17}}A, \quad
\Delta_{A/\Z} = 2^4\cdot 17^2.
$$
We can view $A$ as a rank four $A$-lattice with
$\langle x,y\rangle =\Tr_{A_\Q/\Q}(x\bar{y}z)$,
where 
$$
z= \frac{5+\sqrt{17}}{5-\sqrt{17}} \, \in  \, 
 \, \frac{1}{2}A^+_{\gg 0} \,  \subset  \, 
 \frac{1}{2}A  \, \subset  \, \hat{A}. 
$$
This $A$-lattice has determinant $2^4\cdot 17^2$ 
and is invertible as an $A$-module.
However, it is not
invertible as an $A$-lattice, since $\varphi(1\otimes\bar{1}) = z \notin A$.
\end{ex}

The following lemma, which is used to prove Theorem \ref{LwSthm},
is an analogue of Lemma 11.4 of \cite{LwS}.

\begin{lem}
\label{fractideallem}
If  $A$ is a CM-order and $I$ is an invertible
fractional $A$-ideal, then:
\begin{enumerate}
\item
if $m\in\Z_{>0}$, 
then $I/mI$ is isomorphic to $A/mA$ as $A$-modules;
\item
if $\a \subset A$ is an ideal of finite index, 
then $I/\a I$ is isomorphic to $A/\a$ 
both
as $A$-modules and as $A/\a$-modules;
\item
if $I'$ is a fractional $A$-ideal, then the
natural surjective map 
$$
I\otimes_{A} I' \to II'
$$ 
is an isomorphism.
\end{enumerate}
\end{lem}

\begin{proof}
The proof of (i) is the same as the proof of Lemma 11.4(i) of \cite{LwS}.
Now (ii) follows by letting $m=\#(A/\a)$, so that
$mA \subset \a$, and applying (i)
to show $I/mI \cong A/mA$ as $A$-modules. 
Tensoring with $A/\a$
we have 
$$
I/\a I \cong (I/mI) \otimes_{A/mA}(A/\a)  
\cong (A/mA) \otimes_{A/mA}(A/\a) \cong A/\a
$$
as $A$-modules and as $A/\a$-modules, giving (ii).
The proof of (iii) is the same as the  proof of Lemma 11.4(iii) of \cite{LwS}. 
\end{proof}

\begin{thm}
\label{LwSthm}
Suppose $A$ is a CM-order and $L$ is an $A$-lattice.
Then $L$ is invertible as an $A$-lattice if and only if
there exist a 
fractional $A$-ideal $I \subset A_\Q$  and an element 
$w \in (A^+_\Q)_{\gg 0}$ such that
\begin{enumerate}
\item
$I\bar{I} = Aw$ and
\item
$L$ and $L_{(I,w)}$ are
isomorphic as $A$-lattices.
\end{enumerate}
\end{thm}

\begin{proof}
Suppose there exist a 
fractional $A$-ideal $I \subset A_\Q$  and an element
$w \in (A^+_\Q)_{\gg 0}$ 
satisfying (i) and (ii).
By Lemma \ref{fractideallem}(iii) we have
$$
I \otimes_A \bar{I} \isom I\bar{I} \isom A, \quad x\otimes \bar{y} \mapsto x\bar{y} \mapsto 
\frac{x\bar{y}}{w}.
$$
Thus the composition 
$\varphi: L \otimes_A \bar{L} = I \otimes_A \bar{I}  \isom A$ 
is an isomorphism, and 
$$
\Tr_{A_\Q/\Q}\left(\varphi(x\otimes \bar{y})\right) =
\Tr_{A_\Q/\Q}\left(\frac{x\overline{y}}{w}\right) = 
\langle x,y \rangle_{L_{(I,w)}} = \langle x,y \rangle_{L},
$$
so $\varphi = \varphi_L$. 

Conversely, suppose that $L$ is an invertible $A$-lattice. 
Extending $\Q$-linearly 
the map $\varphi$ from Proposition \ref{latticeequivdefs1}
we have an isomorphism 
$
\varphi : L_\Q \otimes_{A_\Q} \bar{L}_\Q \isom A_\Q
$
as $A_\Q$-modules.
Lemma \ref{Linvremarks} gives that
$L_\Q$ and $A_\Q$ are isomorphic as $A_\Q$-modules,
so we may assume $L_\Q=A_\Q$.
Then $L$ is a finitely generated $A$-submodule of $A_\Q$ spanning
$A_\Q$ over $\Q$, so $L = I$ for some fractional ideal $I$.
We may then take $\bar{L} = \bar{I}$.
The inclusion $I \subset A_\Q$ induces an isomorphism $I_\Q \isom A_\Q$,
which induces an  $A_\Q$-module isomorphism 
$f : I_\Q \otimes_{A_\Q} \bar{I}_\Q \isom A_\Q \otimes_{A_\Q} A_\Q$.
Letting $i$ be the isomorphism 
$
i : A_\Q \otimes_{A_\Q} A_\Q \isom A_\Q$, $x\otimes y \mapsto xy,
$
then the composition $i\circ f\circ \varphi^{-1} : A_\Q \isom A_\Q$ is an
$A_\Q$-module isomorphism
and thus is multiplication by a unit $w\in A_\Q^\ast$. 
So the isomorphism  
$
\varphi : I \otimes_{A} \bar{I} = L \otimes_{A} \overline{L} \isom A$
takes $x\otimes y \in I \otimes_{A} \bar{I}$ to $x\bar{y}/w \in A$,
so $I\bar{I}/w = A$. 

Suppose $x\in I\cap\Q_{>0}$.
Then  
$$
x^2/w = \varphi(x\otimes\bar{x}) = \overline{\varphi(x\otimes\bar{x})}
=\overline{x^2/w} = x^2/\bar{w},
$$ 
so $w=\bar{w}$.
Further, $x^2/w \in A^+_{>0}$, so
for all $\psi\in\Rhom(A_\Q,\C)$ we have $\psi(x^2/w) \in\R_{\ge 0}$, so 
$\psi(w) \ge 0$.
Since $w\in A_\Q^\ast$, for all $\psi\in\Rhom(A_\Q,\C)$ we have $\psi(w) \neq 0$.
Thus $w \in (A^+_\Q)_{\gg 0}$, and
$L$ and $L_{(I,w)}$ are $A$-isomorphic.
\end{proof}

The following result will be used to prove 
Propositions \ref{elemprop2pf} and \ref{ingredient4pf}.

\begin{cor}
\label{egothicacor}
If $A$ is a CM-order, $L$ is an invertible $A$-lattice,
and $\a \subset A$ is an ideal of finite index, then
there exists $e_\a \in L$ such that 
$(A/\a)e_\a = L/\a L.$ 
\end{cor}

\begin{proof}
This follows directly from Theorem \ref{LwSthm} and Lemma \ref{fractideallem}(ii).
\end{proof}

In Algorithm 1.1 of
\cite{modules} we obtained a deterministic polynomial-time algorithm
that on input a finite commutative ring $R$ 
and a finite $R$-module $M$,
decides whether there exists $y\in M$ such that $M = Ry$, and if there is, 
finds such a $y$.
Applying this with $R=A/\a$ and
$M=L/\a L$, gives the  algorithm in the following result,
which is an analogue of Proposition 10.1 of \cite{LwS}.

\begin{prop}
\label{elemprop}
There is a deterministic  polynomial-time algorithm that,
given a CM-order $A$, an $A$-lattice $L$, and an ideal $\a \subset A$
of finite index,  
decides whether there exists $e_\a\in L$ such that 
$
(A/\a)e_\a = L/\a L,
$ 
and if there is, finds one.
\end{prop}

If $L$ is an {\em invertible} $A$-lattice then $e_\a$ exists
by Corollary \ref{egothicacor}.

Recall the definition of vigilant in Definition \ref{vigilantdef}.

\begin{defn}
\label{gooddefn}
Suppose $A$ is a reduced order and
$\a$ is an ideal of $A$. Let
$$
V(\a) = \{ \p\in\Maxspec(A) : \p \supset \a\}. 
$$
We say $\a$ is {\em good} if $\#(A/\a) < \infty$
and $V(\a)$ is vigilant.
\end{defn}

In other words, $\a$ is good if $\#(A/\a) < \infty$
and for all $\r\in\Minspec(A)$ we have $\r + \a \neq A$.

\begin{lem}
\label{mAgood}
If $A$ is a reduced order and $m\in\Z_{>1}$,
then $V(mA)$ is vigilant and $mA$ is good.
\end{lem}

\begin{proof}
Suppose $\r\in\Minspec(A)$.
Then $A/\r$ is an order.
Since $m > 1$ we have $m(A/\r)\neq A/\r$,
so $\r + mA \neq A$. 
The desired result now follows.
\end{proof}

The following result is an analogue of Lemma 10.2 of \cite{LwS}.

\begin{lem} 
\label{reglem2}
Suppose  $A$ is a CM-order,
$L$ is an $A$-lattice, 
and $e\in L$.
\begin{enumerate}
\item
Suppose $m\in\Z_{>1}$. Then
$
(A/mA)e = L/m L
$ 
 if and only if $L/(Ae)$ 
is finite of order coprime to $m$.
\item
Suppose $\rk_\Z(L)=\rk_\Z(A)$ 
and
$L/(A e)$ is finite. Then the map
$
A \to A e$,  $a\mapsto ae
$
is an isomorphism of $A$-modules, i.e., $e$ is regular.
\item
Suppose  
$\a$ is a good ideal of $A$ 
and 
$(A/\a)e = L/\a L$.
Then $L/(A e)$ is finite  
and $L_\Q = A_\Q\cdot e$.
\end{enumerate}
\end{lem}

\begin{proof}
The proof of Lemma 10.2 of \cite{LwS} 
with $\Z\langle G\rangle$ replaced by $A$ shows (i) and (ii).

For (iii),
we have $Ae + \a L = L$, so $\a(L/Ae)=L/Ae$.
By Proposition \ref{Nakayama} (Nakayama's Lemma) there exists 
$x\in 1 + \a \subset A$ such that $x(L/Ae) = 0$.
Since $\a$ is good,
for all $\r\in\Minspec(A)$ we have $\a + \r \neq A$;
thus $1 \notin\a + \r$.
Since $x\in 1 + \a$, it follows that $x\notin \r$ for all $\r\in\Minspec(A)$,
so $x\in A_\Q^\ast$.
Since $x(L/Ae)_\Q = 0$ we have $(L/Ae)_\Q = 0$, so $L/Ae$ is finite
and $L_\Q = A_\Q\cdot e$.
\end{proof}

The following lemma will be used to prove Proposition \ref{elemprop2pf}.
It serves as an analogue of Lemma 11.5 of \cite{LwS}.

\begin{lem}
\label{invlemnew}
Suppose $A$ is a CM-order, $L$ is an 
$A$-lattice, and $\rank_\Z(L)= \rank_\Z(A)$.
Suppose $e_2\in L$ satisfies  
$
(A/2A)e_2 = L/2L,
$ 
and let $z = e_2\overline{e_2} \in A_\Q$
and
$I = \{ a\in A_\Q : ae_2\in L\}$.
Then: 
\begin{enumerate}
\item
$L/(A e_2)$ is finite, $e_2$ is regular, 
$L_\Q = A_\Q e_2$, 
and $L = Ie_2$;
\item
$z \in A_\Q^\ast \cap (A^+_\Q)_{\gg 0}$;
\item
if $L$ is invertible as an $A$-lattice 
and $w = z^{-1}$,
then 
$I\bar{I} = Aw$,  the map $I \to L$, $a \mapsto ae_2$ induces an
$A$-isomorphism from  $L_{(I,w)}$ to $L$,
and $$\varphi_L(x\otimes\bar{y}) = \sigma^{-1}(x)\overline{\sigma^{-1}(y)}z$$
for all $x,y\in L$, where $\sigma : I \isom L$, $a\mapsto ae_2$.
\end{enumerate}
\end{lem}

\begin{proof}
In the notation of  Proposition \ref{latticeequivdefs1} we have
$z = e_2\overline{e_2} = z_{e_2,e_2} = \varphi(e_2 \otimes \overline{e_2})$,
and
$\langle ae_2,ae_2 \rangle = \Tr_{A_\Q/\Q}(a\bar{a}z)$ 
for all $a \in A$.
By Proposition \ref{latticeequivdefs1}(ii)(c) we have
$z \in (A^+_\Q)_{> 0}$.

By Lemma \ref{reglem2} we have that $L/(A e)$ is finite, 
$e_2$ is regular,
the map $A_\Q \to L_\Q$, $a \mapsto ae_2$ is an isomorphism,
and $L_\Q = A_\Q e_2$. 
By the definition of $I$,  
we now have $L = Ie_2$.
This gives (i).

If $a \in A$ and $az=0$, then 
$\langle ae_2,ae_2 \rangle = \Tr_{A_\Q/\Q}(a\bar{a}z) = 0$,
so $ae_2=0$, so $a=0$. Thus multiplication by $z$ is injective, and
therefore surjective, on $A_\Q$.
Thus $z \in A_\Q^\ast$.
Since $z \in (A^+_\Q)_{> 0}$ we now have
$z \in (A^+_\Q)_{\gg 0}$, giving (ii).

Suppose $L$ is invertible. 
Then 
$A = \varphi_L(L\otimes_A \overline{L}) = 
\varphi_L(Ie_2\otimes_A \overline{Ie_2}) = 
I\overline{I}z,$
and $\langle a,b \rangle_{L_{(I,w)}} = 
\Tr_{A_\Q/\Q}(a\bar{b}/w) =  \Tr_{A_\Q/\Q}(a\bar{b}z) = 
\langle ae_2,be_2 \rangle_L$ for all $a,b\in A$.

If $x=ae_2$ and $y=be_2$ with $a,b\in I$, then
$\varphi_L(x\otimes\bar{y}) = x\bar{y} = (ae_2)(\overline{be_2}) = a\bar{b}z$
as desired, giving (iii).
\end{proof}

Algorithm \ref{elemprop2} and Proposition \ref{elemprop2pf} below
extend Algorithm 10.3 and Proposition 10.4 of \cite{LwS}.

\begin{algorithm}
\label{elemprop2}
Given a CM-order $A$ and an $A$-lattice $L$, 
the algorithm decides whether $L$ is invertible,
and if so, outputs the map 
$\varphi : L \otimes_{A} \bar{L} \isom A$ from
Proposition \ref{latticeequivdefs1}.

Steps:
\begin{enumerate}
\item
Check whether $\rank_\Z(L)= \rank_\Z(A)$. 
If it does not, 
output ``no'' and stop.
\item
Run the algorithm associated with
Proposition \ref{elemprop} to decide if there exists
$e_2\in L$ such that 
$
(A/2A)e_2 = L/2L,
$ 
and if so, to compute such an $e_2$. If not, output ``no'' and stop.
\item
Use linear algebra over $\Z$ to 
compute a $\Z$-basis for $I = \{ a\in A_\Q : ae_2\in L\}$. 
\item
Solve for $z\in A_\Q$ in the system of linear equations
$$
\Tr_{A_\Q/\Q}(\alpha_iz) = \langle \alpha_ie_2,e_2\rangle
$$
where $\{ \alpha_i\}_{i=1}^n$ is the $\Z$-basis used for $A$.
\item
Output ``no'' and stop if $I\bar{I}z\neq A$, and otherwise output ``yes'' and 
the map $$\varphi : L \otimes_{A} \bar{L} \to A, \quad
x\otimes\bar{y} \mapsto \sigma^{-1}(x)\overline{\sigma^{-1}(y)}z$$
where $\sigma : I \isom L$, $a\mapsto ae_2$.
\end{enumerate}
\end{algorithm}

\begin{prop}
\label{elemprop2pf}
Algorithm \ref{elemprop2} is correct and
runs in  polynomial time.
\end{prop}

\begin{proof}
If $L$ is invertible, then 
$\rank_\Z(L)= \rank_\Z(A)$ by Lemma \ref{Linvremarks}, and
there exists $e_2$ as in Step (ii) by Corollary \ref{egothicacor}. 

The set $I$ in Step (iii) is clearly a fractional $A$-ideal.
Step (iv) computes  $z\in A_\Q$ such that
$\Tr_{A_\Q/\Q}(x\bar{y}z) = \langle xe_2,ye_2\rangle$
for all $x,y\in I$.
It follows from Proposition \ref{latticeequivdefs1}(i) that  
there is a unique such 
$z$ in ${A_\Q}$,
and 
$z=z_{e_2,e_2} = \varphi(e_2 \otimes \overline{e_2}) = e_2\overline{e_2}$.
By Step (i) and Lemma \ref{invlemnew}, the element $e_2$ is regular,
the map $A_\Q \to L_\Q$, $a \mapsto ae_2$ is an isomorphism
that takes $I$ to $L$, 
and
$z \in A_\Q^\ast \cap (A^+_\Q)_{\gg 0}$.

By Lemma \ref{invlemnew}, if $L$ is invertible, then 
$I\bar{I}z = A$ and Step (v) produces the desired map $\varphi$.
Conversely, if Step (v) determines that $I\bar{I}z= A$,
then the $A$-lattice $L$ is invertible by Theorem \ref{LwSthm}.
\end{proof}

\begin{rem}
To obtain an algorithm that, given a CM-order $A$ and an {\em invertible} 
$A$-lattice $L$, 
outputs $\varphi$, one can simply run Steps (ii)--(v) of Algorithm \ref{elemprop2}
to compute  
the map $\varphi$, without performing the checks
for invertibility.
In the algorithms in this paper, for invertible $A$-lattices
we generally assume  (and suppress mention) that this
has been done, if one needs to perform computations using $\varphi$.
\end{rem}

\section{Short vectors in invertible lattices}
\label{shorttensorsect}

The following theorem generalizes Theorem 12.4 of \cite{LwS}.

\begin{thm}
\label{shortthm} 
Suppose $A$ is a CM-order and $L$ is an $A$-lattice. Then:
\begin{enumerate}

\item
if $L$ is invertible, then the map 
$$
F : \{ \text{$A$-isomorphisms $A \isom L$}\} \to 
\{ \text{short vectors of $L$}\}, \quad \fff \mapsto \fff(1)
$$
is bijective;
\item
if $L$ is invertible and $e \in L$ is short, then 
$e$ generates $L$ as an $A$-module;
\item
 $L$ is $A$-isomorphic to the standard $A$-lattice
if and only if $L$ is invertible and has a short vector;
\item
if $L$ is invertible and $e \in L$ is short, 
then the map 
$$
\mu(A) \to \{\text{short vectors of $L$\}}, \qquad
\zeta \mapsto \zeta e
$$ 
is bijective.
\end{enumerate}
\end{thm}

\begin{proof}
Suppose $L$ is invertible. First suppose $\fff : A \isom L$ is an $A$-isomorphism.
Since $1\in A$ is short, it follows that 
$\fff(1)\in L$ is short.
Thus, $F$ is well-defined.
Since $\fff(a) = \fff(a\cdot 1) = a\fff(1)$, 
the map $\fff$ is determined by $\fff(1)$.
Thus, $F$ is injective. 

For surjectivity of $F$, let $x\in L$ be short and define
$\fff : A \to L$ by $\fff(a) = ax$.
Then $\fff$ is $A$-linear, $\fff(1)=x$, and $\fff$ is injective
(since $x$ is regular by Proposition \ref{shortequiv}).
The map $\fff$ preserves the lattice structure since for all
$a,b\in A$ we have
$$
\varphi_L (\fff(a)\otimes\overline{\fff(b)}) = 
\varphi_L (ax\otimes\bar{b}\bar{x}) = 
a\bar{b}\varphi_L (x\otimes\bar{x}) = 
a\cdot\bar{b} = 
\varphi_A (a\otimes \bar{b}).
$$
To see that $\fff$ is surjective, consider the exact sequences
$$
A\otimes_A \bar{A} \xrightarrow{{\fff}\otimes\id} L\otimes_A \bar{A} \to (\coker(\fff))\otimes_A \bar{A} \to 0
$$
and
$$
L\otimes_A \bar{A} \xrightarrow{\id\otimes\bar{\fff}} L\otimes_A \bar{L} \to L\otimes_A \overline{(\coker(\fff))} \to 0.
$$
Since $L$ is invertible,  
$$
(\id\otimes\bar{\fff})\circ ({\fff}\otimes\id) = 
{\fff}\otimes \bar{\fff}  = (\varphi_L)^{-1}\circ \varphi_A : 
A\otimes_A \bar{A} \to L\otimes_A \bar{L}
$$
is an isomorphism, so
$\id\otimes\bar{\fff}$ is onto.
Thus 
$L\otimes_A \overline{(\coker(\fff))} = 0$,
so 
$$
A\otimes_A {\coker(\fff)} = L\otimes_A \bar{L}\otimes_A {\coker(\fff)} = 0,
$$
so ${\coker(\fff)} = 0$.
This proves (i).

If $L$ is invertible and $e \in L$ is short, then 
$L=Ae$ by (i), and this gives (ii).

For (iii), it suffices to assume that $L$ is invertible, and in that case
(iii) follows from (i).

For (iv), by (iii) we can (and do) reduce to the case where $L$ is the standard $A$-lattice.
By Proposition \ref{muAequiv}, the short vectors are exactly the roots of unity
in $A$. Now (iv) follows easily.
\end{proof}

By Theorem \ref{shortthm}(iii) and (iv), if $L$ is an invertible $A$-lattice and 
$X$ is the set of short vectors in $L$, then
$X = \mu(A)e$
if $L$ is $A$-isomorphic to the standard $A$-lattice and $e\in X$, 
and $X$ is empty otherwise.
Thus, $X$ might be too large to even write down in polynomial
time.

\section{The Witt-Picard group}
\label{WPsect}

As in the introduction, we define
$\WPic(A)$ to be the quotient of
$$
\{(I,w) : I \text{ is an invertible fractional
$A$-ideal, $w \in (A^+_\Q)_{\gg 0}$,  and $I\cdot\bar{I}=Aw$} 
 \}
$$
by $\{(Av,v\bar{v}) : v\in A_\Q^\ast\}$. 
Just as for the class groups in algebraic number theory, 
$\WPic(A)$ is a finite abelian group (Theorem \ref{WPgpfin} below).

The following result is an analogue of
Theorem 13.3, Proposition 13.4, and
Corollary 14.3 of \cite{LwS}, and
can be proved in a similar manner, but now also making use of
Propositions \ref{latticeequivdefs1} and \ref{latticeequivdefs2}.

\begin{prop}
\label{invertequivcor}
Suppose $A$ is a CM-order and 
$L$, $M$, and $N$ are invertible $A$-lattices. 
Then:
\begin{enumerate}
\item 
 $L \otimes_{A} M$ is an invertible $A$-lattice
with the map 
$$
\varphi_{L \otimes_{A} M}: (L \otimes_{A} M) \otimes_A \overline{(L \otimes_{A} M)} \to A
$$
of Proposition \ref{latticeequivdefs1} given by 
$$
\varphi_{L \otimes_{A} M}((x_1 \otimes y_1) \otimes \overline{(x_2 \otimes y_2)}) = 
\varphi_L(x_1 \otimes \overline{x_2}) \cdot \varphi_M(y_1 \otimes \overline{y_2});
$$
\item 
$\overline{L}$ is an invertible $A$-lattice
with the map $\varphi_{\overline{L}} : \overline{L}\otimes_A L \to A$ 
defined by
$$
\varphi_{\overline{L}}(\bar{x}\otimes y) = \varphi_L(y\otimes \bar{x}) = 
\overline{\varphi_L(x\otimes \bar{y})};
$$
\item 
we have the following canonical $A$-isomorphisms:
$$
L \otimes_{A} M \cong M \otimes_{A} L, \quad
(L \otimes_{A} M) \otimes_{A} N \cong 
L \otimes_{A} (M \otimes_{A} N), \quad
L \otimes_{A} A\cong L, \quad
L \otimes_{A} \overline{L} \cong A;
$$
\item
$L$ and $M$ are $A$-isomorphic if and only if 
$L \otimes_{A} \overline{M}$
and $A$ are $A$-isomorphic.
\end{enumerate}
\end{prop}

Note that $\overline{L\otimes_A M}=\overline{L}\otimes_A\overline{M}$
and (canonically)
$$
L\otimes_A M\otimes_A\overline{L}\otimes_A\overline{M} \cong (L\otimes_A\overline{L})\otimes_A(M\otimes_A\overline{M})\cong A.
$$

The following result is an analogue of
Proposition 14.4 and
Theorem 14.5 of \cite{LwS}.

\begin{thm}
\label{WPgpfin}
The set of
invertible $A$-lattices up to $A$-isomorphism
is a finite abelian group and is isomorphic to $\WPic(A)$.
Here, the group operation on (isomorphism
classes of) invertible $A$-lattices is given by
tensoring over $A$,
the unit element is $(A,\varphi_0)$ with
$\varphi_0(x\otimes \bar{y}) = x\bar{y}$, and 
the inverse of $(L,\varphi_L)$ is $(\overline{L},\varphi_{\overline{L}})$.
\end{thm}

\begin{proof}
The proof is a direct generalization of the proofs of 
Proposition 14.4 and
Theorem 14.5 of \cite{LwS}, with Proposition \ref{LLLlem}
serving in the role of Proposition 3.4 of \cite{LwS}.
\end{proof}

Recall the group $\Cl^-(A)$ from the introduction.

\begin{thm}
\label{wpickercoker}
Let $h: \WPic(A) \to \Cl^-(A)$ be the group homomorphism 
sending the class of $(I,w)$
to the class of $I$. Then $2$ annihilates the kernel and cokernel of $h$.
\end{thm}

\begin{proof}
If $[I]\in \Cl^-(A)$, then there exists $v\in A^\ast_\Q$
such that $I\bar{I}=Av$. Then $I\bar{I}=A\bar{v}$,
so $I^2\bar{I}^2=Av\bar{v}$.
Since $v\bar{v} \in (A^+_\Q)_{\gg 0}$ we have $[(I^2,v\bar{v})] \in \WPic(A)$,
and $h([(I^2,v\bar{v})]) = [I]^2$.

If $[(I,w)]$ is in the kernel of $h$, then there exists $v\in A^\ast_\Q$
such that $I=Av$. 
Since $I\bar{I}=Aw$, it follows that $\bar{I}=Aw/v$.
Since $\bar{w}=w$ we have ${I}=Aw/\bar{v}$.
Thus, ${I}^2=Au$ where $u=wv/\bar{v}\in A^\ast_\Q$.
We now have $(I^2,w^2)= (Au,u\bar{u})$.
\end{proof}

The following proposition summarizes the algorithmic results for $\WPic(A)$
that are proved in the present paper.

\begin{prop}
There are deterministic polynomial-time algorithms for
finding the unit element,
inverting,
multiplying,
exponentiation,
and
equality testing in $\WPic(A)$.
\end{prop}

Algorithms for the unit element and inverting follow easily
from  Theorem \ref{WPgpfin}. 
Multiplication and exponentiation are dealt with in the next section.
See Theorem \ref{mainthmcor} for equality testing.

\section{Multiplying and exponentiating invertible $A$-lattices}
\label{multexpinvlatsect}

This section generalizes Section 15 of \cite{LwS}.
All $A$-lattices in the inputs and
outputs of the algorithms are specified via an LLL-reduced basis.
Direct generalizations of
Algorithms 15.1, 15.2, and 15.3 of \cite{LwS} give the following 
(relying on Lemma \ref{invlemnew} above wherever \cite{LwS} relied
on Lemma 11.5 of \cite{LwS}).

\begin{thm}
\label{LMmultalg}
\begin{enumerate}[leftmargin=*]
\item
There is a deterministic polynomial-time algorithm that,
given a CM-order $A$ of rank $n$
and invertible $A$-lattices $L$ and $M$, outputs 
$L \otimes_{A} M$ 
and an $n\times n\times n$ array of integers to describe the
multiplication map 
$L \times M \to L \otimes_{A} M.$
\item
There is a deterministic polynomial-time algorithm that,
given a CM-order $A$, 
an ideal $\a$ of $A$ of finite index, 
invertible $A$-lattices $L$ and $M$, and
elements $d\in L/\a L$ and $f\in M/\a M$, computes 
$
L \otimes_{A} M
$ 
and the element
$
d\otimes f\in (L \otimes M)/\a (L \otimes M).
$
\item
There is a deterministic polynomial-time algorithm that,
given a CM-order $A$, 
a positive integer $r$, 
an invertible $A$-lattice $L$,
an ideal $\a$ of $A$ of finite index, 
and $d\in L/\a L$, outputs 
$L^{\otimes r}$ and 
$d^{\otimes r}\in L^{\otimes r}/\a L^{\otimes r}$. 
\end{enumerate}
\end{thm}

\section{The extended tensor algebra $\Lambda$}
\label{tensorsect}
We next define
the extended tensor algebra $\Lambda$, which is a single algebraic structure
that comprises all rings and lattices that our main algorithm needs. 
Suppose $A$ is a CM-order and $L$ is an invertible $A$-lattice.
Let $L^{\otimes 0} = A$,  and for all $m \in \Z_{>0}$ let
$$L^{\otimes m} =  
{L} \otimes_{A} \cdots \otimes_{A} {L} 
\quad \text{(with $m$ $L$'s)},$$  
and 
$$L^{\otimes (-m)} = \overline{L}^{\otimes m}= 
\overline{L} \otimes_{A} \cdots 
\otimes_{A} \overline{L}.$$ 
For simplicity, we denote $L^{\otimes m}$ by $L^m$.
If $m\in \Z$, then $L^{m}$ is an invertible $A$-lattice
by Proposition \ref{invertequivcor}, 
and 
$\overline{L^m} = \overline{L}^m = L^{-m}.$

Let 
$$
\Lambda = T_A(L)  = \bigoplus_{i\in\Z} L^{i},
$$ 
an $A$-algebra
with involution $\bar{\,\,}$.
The following result is analogous to Proposition 16.1 of \cite{LwS},
and its proof is straightforward.

\begin{prop}
\label{LambdaProp}
Suppose $A$ is a CM-order and $L$ is an invertible $A$-lattice. Then:
\begin{enumerate} 
\item 
the extended tensor algebra 
$\Lambda$ is a commutative ring containing $A$ as a subring;
\item 
for all  $j\in \Z$, the action of $A$ on $L^j$ becomes multiplication in $\Lambda$;
\item 
$\Lambda$ has an involution $x\mapsto \overline{x}$ extending both the involution
of $A$ and the map $L \isom \overline{L}$;
\item 
if $j\in \Z$, then the map 
$L^j \times \overline{L^j} \to L^j \otimes_A \overline{L^j} \isom A$
induced by the isomorphism $\varphi_{L^j}$  
becomes
multiplication in $\Lambda$, with $\overline{L^j}={L}^{-j}$;
\item 
if $j\in \Z$ and $e \in L^j$ is short, then $\overline{e} = e^{-1}$ in $L^{-j}$;
\item
if $e \in L$ is short, then 
$\Lambda = A[e,e^{-1}],$
where the right side is the subring of $\Lambda$ generated by
$A$, $e$, and $e^{-1}$, which is a Laurent polynomial ring.
\end{enumerate}
\end{prop}

In \cite{Kronecker} we show the following result, which we will use
in Proposition \ref{Bgraded} below. (In \cite{Kronecker}, the group
 $\Gamma$ was written multiplicatively.)

\begin{prop}[\cite{Kronecker}, Theorem 1.5(ii,iii)]
\label{ingredient2}
Suppose $B = \bigoplus_{\gamma\in \Gamma} B_\gamma$ 
is an order that is graded by an additively written finite abelian group $\Gamma$
(i.e.,  
the additive subgroups $B_\gamma$ of $B$ satisfy
$B_\gamma\cdot B_{\gamma'} \subset B_{\gamma+\gamma'}$
for all $\gamma,\gamma'\in \Gamma$, and
the additive group homomorphism  
$\bigoplus_{\gamma\in \Gamma} B_\gamma \to B$ sending 
$(x_\gamma)_{\gamma\in \Gamma}$ to $\sum_{\gamma\in \Gamma}x_\gamma$
is bijective).
Suppose $B_0$ is connected.
Then  $B$ is connected and 
$\mu(B) \subset \bigcup_{\gamma\in \Gamma} B_\gamma$.
\end{prop}

The following result is analogous to Proposition 16.2 of \cite{LwS},
and will be used in Proposition \ref{almostmainalgorpf}.

\begin{prop}
\label{Bgraded}
Suppose $A$ is a CM-order, $L$ is an invertible $A$-lattice,
$r\in\Z_{>0}$, $y\in L^{r}$, and $y\bar{y}=1$.
Let $\Lambda = T_A(L)$ and  
 $B = \Lambda/(y -1)\Lambda$.  
Then:
\begin{enumerate}
\item
the map $\bigoplus_{i=0}^{r-1} L^i \to B$ induced by the natural map
$\bigoplus_{i=0}^{r-1} L^i \subset \Lambda \to B$ is an 
$A$-module isomorphism that exhibits the commutative ring $B$ as  
a $\Z/r\Z$-graded order;
\item
$B$ is a CM-order, with involution $\bar{\,\,}$ on $B$ induced by 
the involution $\bar{\,\,}$ on $\Lambda$;
\item
$\mu(B) = \{ \beta\in B : \beta\bar{\beta}=1\}$;
\item
if $A$ is connected, then $B$ is connected and 
$\mu(B) \subset \bigcup_{i=0}^{r-1} L^{i}$
(identifying $B$ with $\bigoplus_{i=0}^{r-1} L^i$ as in \rm{(i)}).
\end{enumerate}
\end{prop}

\begin{proof}
Part (i) is a straightforward exercise.

Each $L^{i}$ has an $A$-lattice structure
$\langle x,y \rangle = \Tr_{A/\Z}(x\bar{y})$, 
where $x\overline{y} = \varphi_{L^i}(x \otimes \bar{y})$.
If 
$\beta = (\beta_0,\ldots,\beta_{r-1}) \in \bigoplus_{i=0}^{r-1} L^i = B$,
then $\overline{\beta_i} \in L^{-i}$ and
$y\overline{\beta_i} \in L^{r-i}$, 
but $\overline{\beta_i} = y\overline{\beta_i}$ in $B$, so 
$$
\bar{\beta} = (\overline{\beta_0},\overline{\beta_{r-1}},\ldots,\overline{\beta_1})
= (\overline{\beta_0},y\overline{\beta_{r-1}},\ldots,y\overline{\beta_1}) 
\in \bigoplus_{i=0}^{r-1} L^i = B.
$$
By Proposition \ref{involCMprop} applied to $E = B_\Q$
and Remark \ref{CMordiffresult},
to prove (ii)
it suffices to prove that for all
$\beta$ 
we have
$\Tr_{B/\Z}(\beta\bar{\beta}) = r\cdot\sum_{i=0}^{r-1} \langle \beta_i,\beta_i \rangle$.
If $a\in A = L^{0}$, then 
$$
\Tr_{B/\Z}(a) = \sum_{i=0}^{r-1} \tt(\text{action of $a$ on $L^{i}$})
= r\cdot\Tr_{A/\Z}(a).
$$
If $c\in  L^{i}$ with $0 < i < r$, then 
$\Tr_{B/\Z}(c) = 0.$ Thus, 
$\Tr_{B/\Z}(\beta) = r\cdot\Tr_{A/\Z}(\beta_0)$.
If $\beta\bar{\beta} = (\alpha_i)_{i=0}^{r-1}$, then $\alpha_0 = \sum_{i=0}^{r-1} \beta_i\overline{\beta_i}$.
Thus,
$$
\Tr_{B/\Z}(\beta\bar{\beta}) = r\cdot\Tr_{A/\Z}(\alpha_0) = 
r\cdot\Tr_{A/\Z}(\sum_{i=0}^{r-1} \beta_i\overline{\beta_i}) = 
r\cdot\sum_{i=0}^{r-1} \Tr_{A/\Z}(\beta_i\overline{\beta_i}) = 
r\cdot\sum_{i=0}^{r-1} \langle \beta_i,\beta_i \rangle.
$$
This proves (ii).
Part (iii) follows from Proposition \ref{muAequiv} and (ii).
Part (iv) follows from 
Proposition \ref{ingredient2} with $\Gamma = \Z/r\Z$, where $B_0 = L^0 = A$.
\end{proof}

The algorithm associated to the following result is Algorithm 13.2 of \cite{RoU}.

\begin{prop}[\cite{RoU}, Theorem 1.2]
\label{ingredient1}
There is a deterministic polynomial-time algorithm that, given an order $B$,
produces a set $S$ of generators for the group $\mu(B)$ of roots of unity in
$B^\ast$, as well as a set of defining relations for $S$.
\end{prop}

The following algorithm will be applied repeatedly in Algorithm \ref{connmainalgor}.
It generalizes Algorithms 17.4 and 19.1(vii)--(ix) of \cite{LwS}. 

\begin{algorithm}
\label{almostmainalgor}
Given a connected CM-order $A$ of rank $n$, 
an invertible $A$-lattice $L$,
a positive integer $r$,  
an element
$\epsilon \in L/2^{n+1}L$ 
such that $(A/2^{n+1}A)\epsilon = L/2^{n+1}L$,
and an element $s\in A/2^{n+1}A$ such that 
the coset
$s\epsilon^r
\in L^r/2^{n+1}L^r$ contains a (unique)
short vector, 
the algorithm 
decides whether $L$ has a short vector,  
and if so,  
determines an element $t\in A/2^{n+1}A$ such that
the coset $t\epsilon$ 
contains a (unique) short vector.

Steps:
\begin{enumerate}
\item
Pick an element $\ee$ in the coset $\epsilon$ 
and let $q = (L: A\ee)$.
Apply 
the algorithm associated to Proposition \ref{elemprop} 
to find $e_q\in L$ such that $Ae_q + qL = L$. 
Let $I = \{ a\in A_\Q : a\ee \in L\}$
and $w = (\ee\overline{\ee})^{-1} \in A_\Q^\ast$, 
compute $w^r$, 
compute $\beta = e_q/\ee \in A_\Q \subset \Lambda_\Q$,
and for $0\le i \le r$ 
compute $I^i = A + A\beta^i$.
\item
Apply Algorithm \ref{findingekmalg2} with $\a = 2^{n+1}A$ and $L = L_{(I^r,w^r)}$
and $C = s + 2^{n+1}L_{(I^r,w^r)}$ to compute 
the unique short vector $\nu \in C$.
\item
Construct the order $B = \bigoplus_{i=0}^{r-1} I^i$ with 
multiplication 
$$
I^i\times I^j \to I^{i+j}, \quad (x,y)\mapsto xy
\quad \text{if $i+j<r$}
$$ 
and 
$$
I^i\times I^j \to I^{i+j-r}, \quad (x,y)\mapsto xy/\nu
\quad \text{if $i+j\ge r$}.
$$
\item
Apply the algorithm from   
Proposition \ref{ingredient1}
to compute a set
of generators $\{ \zeta_1,\ldots,\zeta_m \}$ for $\mu(B)$.
\item
Applying the degree map $\deg : \mu(B) \to \Z/r\Z$
that takes $\zeta\in \mu(B)$ to $j\in\Z$ such that $\zeta\in I^{j}$, 
either find integers $s_i$ such that $\sum_{i=1}^m s_i\deg(\zeta_i) = 1$,
or if no such integers exist
output ``no'' and stop.
Letting $\alpha = \prod_{i=1}^m \zeta_i^{s_i} \in \mu(B)$, 
use linear algebra over $\Z$ to  compute 
$t\in A/2^{n+1}A$ that maps to $\alpha$ mod $2^{n+1}I$
under the isomorphism $A/2^{n+1}A \isom I/2^{n+1}I$ induced by $a \mapsto a$,
and output ``yes'' and $t$.
\end{enumerate}
\end{algorithm}

\begin{prop}
\label{almostmainalgorpf}
Algorithm \ref{almostmainalgor} is correct and runs in  
time at most polynomial in $r$ plus the length of the input. 
\end{prop}

\begin{proof}
By  Lemma \ref{reglem2}(i)  
with $m=2^{n+1}$ we have that $q < \infty$. 
Then $e_q$ exists by Corollary \ref{egothicacor}.
By Lemma \ref{invlemnew}
we have that $w \in A_\Q^\ast$ and that the map $I \to L$, $a \mapsto a\ee$
induces an $A$-isomorphism from $L_{(I,w)}$ to $L$.
That $I^i = A + A\beta^i$ 
follows exactly as in the proof of Proposition 19.2 of \cite{LwS},
with $A$ in place of $\Z\langle G\rangle$ and making use of 
Lemma \ref{reglem2}(i).

The short vector $\nu$ in Step (ii) is unique by Proposition \ref{findingekmprop2},
and $\nu \ee^r \in L^r$ is the unique short vector in the coset 
$s\epsilon^r$.

By Proposition \ref{Bgraded}(iv), the degree map in Step (v) makes sense.
Since $\deg(\alpha)=1$ we have $\alpha\in I$. 
Since 
$A\ee + 2^{n+1}L = L$, we have  
$A + 2^{n+1}I = I$, and it follows that the map 
$A/2^{n+1}A \to I/2^{n+1}I$ induced by $a \mapsto a$ is an isomorphism.
By Proposition \ref{Bgraded}(iii), 
the vector $z = \alpha \ee \in L$ satisfies $z\bar{z}=1$, and 
is the unique short vector
in the coset  
$t\epsilon$ 
by Proposition \ref{findingekmprop2}.

Computing  
$w^r$ and $\beta^i$ in Step (i),
and all computations involving $B$,
 entail the $r$ entering the runtime. 
\end{proof}

\section{Some elementary number theory}
\label{Hope2sect}

\begin{defn}
\label{cbndefn}
Let $c(n)= n^2$ for $n\ge 2$, let 
$b(n) = 4(\log~n)^2$ for $n\ge 3$,
and let $c(1)=b(1)=2$ and $b(2)=3$.
\end{defn}
Note that $b(n) \le c(n)$, and $c$ and $b$ are
each monotonically increasing.
Let
$$
\psi(x,y) = \#\{ m\in\Z : 0 < m\le x, \text{ each prime $p|m$ satisfies $p\le y$}\}.
$$

\begin{thm}[Konyagin-Pomerance, Theorem 2.1 of \cite{KP}]
\label{Abtien0}
If $x \ge 4$ and $2 \le y \le x$, then
$
\psi(x,y) > x^{1-\log\log x/\log y}. 
$
\end{thm}

\begin{cor}
\label{Abtien}
For all $n\in\Z_{>0}$ we have 
$$
\psi(c(n),b(n)) > n. 
$$
\end{cor}

\begin{proof}
For $n>2$ this follows 
by setting $x = n^2$ and  $y = 4(\log~n)^2$
in Theorem \ref{Abtien0}. 
For $n=1$ and $2$ this can be checked by hand.
\end{proof}

\begin{prop} 
\label{Hope2}
For each $n\in\Z_{>0}$, each prime divisor of 
$$
\gcd\{ \h^n-1 : \h\in\Z_{>0}, \quad \h \le b(n) \}
$$
is less than $c(n)$.
\end{prop}

\begin{proof}
Suppose $\ell$ is a prime divisor of 
$\gcd\{ \h^n-1 : \h\in\Z, \h \le b(n) \}$.
Then $\h^n \equiv 1\bmod{\ell}$ for all integers $\h \le b(n)$.
Let
$S$ denote the set of $m\in\Z_{>0}$ with $m\le c(n)$ such that
all prime divisors  $p$ of $m$ satisfy $p \le b(n)$.
Then $\# S = \psi(c(n),b(n)) > n$ by Corollary \ref{Abtien},
and 
for all $a\in S$ we have $a^n \equiv 1\bmod{\ell}$.
So if all elements of $S$ are pairwise incongruent mod $\ell$,
then
$\#\{ x\in\F_\ell : x^n=1\} \ge \# S > n$,
which cannot be. So there exist $s,t\in S$ with $s\neq t$ 
and $s\equiv t\bmod{\ell}$.
Thus, $\ell$ divides $|s-t|$, and $|s-t| \le c(n) -1$.
\end{proof}

\begin{cor} 
\label{Hope2Cor}
Suppose 
$n\in\Z_{>0}$ and $\ell$ is a prime number such that $\ell > c(n)$.
Then there exists a prime number
$p \le b(n)$ such that $p^n \not\equiv 1\bmod \ell$.
\end{cor}

\begin{proof}
By Proposition \ref{Hope2}, there exists a positive integer
$h \le b(n)$ such that $h^n \not\equiv 1\bmod \ell$.
Then $h$ has a prime divisor $\zzzz \le b(n)$ such that 
$\zzzz^n \not\equiv 1\bmod \ell$.
\end{proof}

The next result replaces our use of Linnik's theorem in \cite{LenSil,LwS},
and allows us to prove upper bounds for the runtime that are much better
than those proved in \cite{GS,LenSil,LwS,Kirchner}.
We use it to prove Proposition \ref{usablesetsalgpf}.

\begin{prop} 
\label{Hope1}
Suppose 
$A$ is an order and $n=\rk_\Z(A)\in\Z_{>0}$.
Then for each prime number $\ell > c(n)$ there is a maximal ideal $\L$ of $A$ 
such that 
$\Norm(\L) \not\equiv 1\bmod{\ell}$ and
$\fchar(A/\L)\le b(n)$.
\end{prop}

\begin{proof}
By Corollary \ref{Hope2Cor}, there exists a 
prime number $\zzzz \le b(n)$ such that 
$\zzzz^n \not\equiv 1\bmod \ell$.
Take a sequence of ideals
$$
\a_0 = A \supsetneq \a_1 \supsetneq \a_2 \cdots  \supsetneq \a_m = \zzzz A
$$
such that each $\a_{i-1}/\a_i$ is a simple $A$-module.
Then $\a_{i-1}/\a_i \cong A/\L_i$ as $A$-modules, 
for some maximal ideal $\L_i$ of $A$ with $\fchar(A/\L_i)=p$.
Now,
$$
\prod_{i=1}^m \N(\L_i) = 
\prod_{i=1}^m \#(A/\L_i) = 
\prod_{i=1}^m\#(\a_{i-1}/\a_i) = 
\#(A/\zzzz A) = 
\zzzz^n \not\equiv 1\bmod{\ell}. 
$$
Thus $\N(\L_i) \not\equiv 1\bmod{\ell}$ for some $i$. 
\end{proof}

We now give a heuristic argument that  
$b(n)$ can be replaced with $5$
in Corollary~\ref{Hope2Cor}.
If $\ell$ is a prime 
let $G_\ell = \langle 2,3,5 \bmod{\ell}\rangle \subset (\Z/\ell\Z)^\times$, and 
if $m\in\Z_{>0}$ let
$$
T_m = \{ \text{primes } \ell \, : \, \ell > m^2, \, \ell > 5, \text{ and } m = \#G_\ell \}.
$$
If $b(n)$ cannot be replaced with $5$ in Corollary~\ref{Hope2Cor},
then there exists $n\in\Z_{>0}$ and a prime number $\ell > c(n) \ge n^2$
such that 
$p^n \equiv 1\bmod \ell$ for all $p \in \{2, 3,5\}$; 
if $g$ is a generator of the cyclic group $G_\ell$,
then $g^n \equiv 1\bmod \ell$, so if $m = \#G_\ell$ then
$m$ divides $n$  
and it follows that 
$\ell \in T_m$.
Thus it would suffice to show that $T_m$ is empty for all $m\in\Z_{>0}$.
Let $T_{m,x} = \{ \ell \in T_m : \ell > x\}$.
If $\ell\in T_{m,x}$ then we can write $\ell = km+1$ with 
$k\in \Z$ and $k \ge m$ (since $\ell > m^2$) and $k\ge x/m$ (since $\ell > x$).
Heuristically, a given pair $(k,m)$ gives an $\ell\in T_m$ with
``probability'' at most $c/k^3$ with an absolute positive constant $c$, since 
the probability that $\ell$ is prime is at most $1$ and
the probability that $2^m \equiv 3^m \equiv 5^m \equiv 1\bmod \ell$ 
once $\ell$ is prime might naively be 
estimated as $1/k^3$, with
the constant $c$ accounting for effects coming from quadratic
reciprocity. 
So one ``expects'' the set $\bigcup_{m \ge 1} T_{m,x}$ to have size at most
\begin{multline*}
c\sum_{m \ge 1}\left(\sum_{k\ge\max\{m,x/m\}} \frac{1}{k^3}\right) \le 
c'\sum_{m \ge 1} \frac{1}{\max\{m,x/m\}^2} = \\
c'\left(\sum_{m \le \sqrt{x}} \frac{m^2}{x^2} + \sum_{m > \sqrt{x}} \frac{1}{m^2}\right) \le
c''\left(\frac{\sqrt{x}^3}{x^2} +  \frac{1}{\sqrt{x}}\right)
= \frac{2c''}{\sqrt{x}},
\end{multline*}
which is less than $1$ for all sufficiently large $x$.
For all primes $\ell$ from $7$ to $100$ million,
we easily check that $\ell < m^2 = (\#G_\ell)^2$
(in fact, $\ell < (\#\langle 2,3 \bmod{\ell}\rangle)^{1.85}$),
so $\ell \notin T_m$. 
Similarly, $b(n)$ can be replaced with $5$, heuristically,
in Proposition~\ref{Hope1} and Theorem \ref{newLinnik}.
However, if one deletes the $5$ in the definition of $G_\ell$,
then conjecturally infinitely many $T_m$ are non-empty,
by essentially the above argument, but not a single such $m$ is known.

\section{Finding auxiliary ideals}
\label{Hope2sectbis}

Corollary 2.8 of \cite{Ronyai} gives a polynomial-time algorithm
that on input a prime $p$ and a finite dimensional commutative $\F_p$-algebra
(specified by structure constants), computes its nilradical.
Corollary 3.2 of \cite{Ronyai} gives an algorithm
that on input a prime $p$ and a finite dimensional semisimple commutative $\F_p$-algebra
$R$, computes its minimal ideals in
time at most polynomial in $p$ plus $\dim_{\F_p}(R)$. 
Combining these gives the following result.

\begin{thm}[\cite{Ronyai}]
\label{Ronyaialgs}
There is an algorithm that on input 
a prime $p$ and a finite dimensional commutative $\F_p$-algebra $R$,
computes the prime ideals of $R$ in time at most polynomial in $p$ plus 
the length of the input.
\end{thm}

Recall the definition of vigilant from Definition \ref{vigilantdef}
and the functions $b$ and $c$ from Definition \ref{cbndefn}.

\begin{defn}
\label{usabledef}
Suppose $A$ is a reduced order of rank $n > 0$.
We will call a set $\SS$ {\em usable} for $A$ if 
$\SS$ consists of vigilant sets $S$ for $A$ such that:
\begin{enumerate}
\item
$\fchar(A/\p) \le b(n)$ for all $S\in \SS$ and all $\p\in S$, 
\item
for each prime number $\ell > c(n)$ there exists $S\in\SS$ such that 
for all $\L\in S$ we have $\N(\L) \not\equiv 1 \bmod \ell$, and
\item
the set 
$$S_0 = \{ \p\in\Maxspec(A) : 2 \in\p\}
$$
belongs to $\SS$.
\end{enumerate}
\end{defn}

If $\r\in\Minspec(A)$, let 
$
d_\r = \rank_\Z(A/\r).
$

The next algorithm will be used in Algorithm \ref{Usetalgor}.

\begin{algorithm}
\label{usablesetsalg}
Given a  
reduced order $A$ of rank $n>0$,
the algorithm outputs a  finite set $\SS$ that is usable for $A$.

Steps:
\begin{enumerate}
\item
Apply Algorithm 7.2 of \cite{Qalgs} to find $\Minspec(A) = \Spec(A_\Q)$.
\item
Find  $\beta\in\Z$ such that $|\beta - \max_{\r\in\Minspec(A)} b(d_\r)| < 1$.
\item
For each prime number $p \le \beta$ apply 
the algorithm associated with Theorem \ref{Ronyaialgs}
to find all prime ideals of the finite commutative
$\F_p$-algebra $A/pA$, i.e., 
find the set $\MM$ of all $\L\in\Maxspec(A)$ such that 
$\fchar(A/\L) \le \beta$.
For each  $\L\in\MM$, mark which $\r\in\Minspec(A)$ 
satisfy $\r\subset\L$.
\item
With input the finite set
$$
\tilde{S} = \{ \text{primes $\ell \le c(n)$} \}
\cup \{ \Norm(\L)-1 : \L\in \MM \} 
\subset \Z_{>0},
$$
apply the Coprime Base Algorithm from \cite{CBA}
to obtain a finite set $T\subset \Z_{>1}$ and a map
$e : \tilde{S} \times T \to \Z_{\ge 0}$ such that
\begin{enumerate}
\item
for all $t,t'\in T$ with $t\neq t'$ we have $\gcd(t,t')=1$, and
\item
for all $s\in \tilde{S}$ we have $s = \prod_{t\in T}t^{e(s,t)}$.
\end{enumerate}
\item
Define a set $T'$ of integers coprime to all primes $\ell \le c(n)$  by 
$$
T = T' \amalg   \{ \text{primes $\ell \le c(n)$} \}.
$$
For all $\L\in\MM$ and 
$t\in T$, 
define $h_\L(t) = e(\N(\L) - 1,t) \in \Z_{\ge 0}$, i.e.,
$$\N(\L) - 1 = \prod_{t\in T}t^{h_\L(t)}.$$
With $S_0$ as in Definition \ref{usabledef}(iii), define  
$$T'' = \{ t\in T' : \max\{ h_\L(t) : \L\in S_0\} > 0 \}.$$
If 
$T''$ is empty, output
$\SS = \{ S_0\}$ and stop. 
Otherwise, proceed as follows.
\item
For each $t\in T''$  
and each $\r\in \Minspec(A)$,
find $\L_{t,\r}\in\MM$ from Step (i) such that
$h_{\L_{t,\r}}(t)=0$ and $\r\subset \L_{t,\r}$. 
Let 
$$
S_t = \{ \L_{t,\r} : \r\in \Minspec(A) \} \subset \MM
$$
and output 
$$\SS = \{ S_0\} \cup \{ S_t : t\in T'' \}.$$
\end{enumerate}
\end{algorithm}

\begin{prop}
\label{usablesetsalgpf}
Algorithm \ref{usablesetsalg} is correct and runs in  
polynomial time. 
\end{prop}

\begin{proof}
That Step (vi) can find, for each $t\in T''$ and each $\r\in\Minspec(A)$,
a maximal ideal $\L_{t,\r}$ in $A$ such that
$h_{\L_{t,\r}}(t)=0$ and $\r\subset \L_{t,\r}$ can be seen as follows.
Since $t\in T'' \subset T$ we have $t>1$.
Suppose $\ell$ is a prime divisor of $t$. Since $t\in T'$, we have
$\ell > c(n) \ge c(d_\r)$
so by Proposition \ref{Hope1} applied with $A/\r$ in place of $A$
there is a maximal ideal $\L_{t,\r}$ of $A$ that contains $\r$ such that  
$\fchar(A/\L_{t,\r}) \le b(d_\r)$  and $\N(\L_{t,\r}) \not\equiv 1 \bmod \ell$.
Thus $\N(\L_{t,\r}) \not\equiv 1 \bmod t$,  
so $h_{\L_{t,\r}}(t)=0$.

The sets $S_t$ for $t\in T''$ were constructed to be vigilant. 
The set $S_0$ is vigilant by Lemma \ref{mAgood} with $m=2$.  

To see that $\SS$ is usable, first note that
if $S\in \SS$ and $\p\in S$ then
$\fchar(A/\p) \le \beta \le b(n)$.
Let $\ell$ be a prime number $> c(n)$.
We will  
show that there exists $S\in\SS$ such that 
for all $\L\in S$ we have
$\Norm(\L) \not\equiv 1\bmod{\ell}$.
If $\ell$ divides some $t\in T''$,
then take $S=S_t$. 
If $\ell$ does not divide any element of $T''$, take $S = S_0$.

Step (i) runs in polynomial time by Theorem 1.10 of \cite{Qalgs}.

The primes $p$ in Step (iii) are so small in size and number that
the appeals to Theorem \ref{Ronyaialgs} run in time at most polynomial in the
length of the input specifying $A$.

Step (iv) runs in polynomial time 
since the Coprime Base Algorithm in \cite{CBA} does.
Steps (v) and (vi) run in polynomial time since $T''$ is a subset of
$T$, which was computed via a polynomial-time algorithm.
\end{proof}

\section{Using the auxiliary ideals}
\label{applyauxprssect}

Recall the definition of $S_0$ in Definition \ref{usabledef}(iii).

\begin{defn}
\label{kadefn}
Suppose $A$ is an order of rank $n>0$. If  
$\a = \prod_{\p\in \Maxspec(A)} \p^{t_\a(\p)}$ is an ideal in $A$
with $t_\a(\p)\in\Z_{\ge 0}$, and $\a \neq 2^{n+1}A$,
let
$$
k(\a) = \lcm_{\p} \{ (\N(\p) -1)p_\p^{t_\a(\p)-1}\}
$$
where $p_\p = \fchar(A/\p)$ denotes the prime number in $\p$, 
and $\N(\p) = \#(A/\p)$,
and the $\lcm$ is over the maximal ideals $\p$ with $t_\a(\p)\in\Z_{> 0}$.
Let
$$
k({2^{n+1}A}) = \frac{2^{2n} \lcm_{\p\in S_0} \{ \N(\p) -1\}}{\prod_{\p\in S_0}\N(\p)}.
$$
\end{defn}

The number $k(\a)$ is the analogue of the number
$k(m)$ that was defined in Notation 18.1 of \cite{LwS} 
for positive integers $m$.

\begin{lem}
\label{exponAa}
Let $A$ be an order of rank $n>0$.
The exponent of $(A/2^{n+1}A)^\ast$ divides
$k({2^{n+1}A})$ 
and is less than $2^{2n}$.
If $\a = \prod_{\p\in \Maxspec(A)} \p^{t_\a(\p)}$ is an ideal in $A$
with $t_\a(\p)\in\Z_{\ge 0}$, then 
the exponent of the group $(A/\a)^\ast$ divides 
$k(\a)$.
\end{lem}

\begin{proof}
Let $G = (A/2^{n+1}A)^\ast$, let $\ccc = \bigcap_{\p\in S_0} \p$,
let $U_0$ be the kernel of the natural map $G \to (A/\ccc)^\ast$, and
for $i \in \{ 1,\ldots,n+1 \}$  
let $U_i$ be the kernel of the natural map $G \to (A/2^iA)^\ast$.
We have
$
G \supset U_0 \supset U_1 \supset \cdots \supset U_{n+1} = 1.
$
Further, 
$$
G/U_0 \cong (A/\ccc)^\ast \cong \prod_{\p\in S_0}(A/\p)^\ast,
$$
which has exponent $\lcm_{\p\in S_0} \{ \N(\p) -1\}$.
Since $U_{0}/U_1 \cong 1 + \ccc/2A$,
we have 
$$
\#(U_{0}/U_1) = \#(\ccc/2A) = \frac{2^n}{\prod_{\p\in S_0}\N(\p)}.
$$ 
For $1 \le i \le n$, the group $U_i/U_{i+1}$ has exponent $2$.
Thus the exponent of $G$ divides
$k({2^{n+1}A})$.
Since $G/U_1 \cong (A/2A)^\ast$, 
the exponent of $G$ is less than
$2^n\#(A/2A) = 2^{2n}$.

For the final result, 
suppose $\p\in \Maxspec(A)$ and $t=t_\a(\p) > 0$.
Now let $U_0 = (A/\p^t)^\ast$ and 
for $i \in \{ 1,\ldots,t \}$  
let $U_i$ be the kernel of the natural map $(A/\p^t)^\ast \to (A/\p^i)^\ast$.
Then $U_0 \supset U_1 \supset \cdots \supset U_t = 1$, so the
exponent of $U_0$ divides the product of the
exponents of the groups $U_{i-1}/U_i$ for $i=1,\ldots,t$.  
The exponent of $U_{0}/U_1$ is $\#((A/\p)^\ast) = \N(\p) -1$.
For $i>1$ the exponent of $U_{i-1}/U_i$ is $p_\p$.
Thus the exponent of $(A/\p^t)^\ast$ divides 
$(\N(\p) -1)p_\p^{t-1}$.

Applying the Chinese Remainder Theorem
to the coprime ideals $\p^{t_\a(\p)}$ for which $t_\a(\p)> 0$,  
we have a ring isomorphism 
$A/\a \isom \prod_{\p | \a}A/\p^{t_\a(\p)}$.
It follows that the exponent of $(A/\a)^\ast$ divides the lcm
of the exponents of the groups $(A/\p^{t_\a(\p)})^\ast$, 
which combined with the previous paragraph proves the last
result.
\end{proof}

Recall the definitions of ``good'' from Definition \ref{gooddefn}
and of $c(n)$ from Definition \ref{cbndefn}.
The next algorithm will be invoked in Algorithm \ref{ingredient4}.

\begin{algorithm}
\label{Usetalgor}
Given a connected CM-order $A$ of rank $n$, 
the algorithm outputs:
\begin{itemize}
\item 
a finite set $U$ of good ideals 
$\a$ of $A$ such that $2^{n+1}A\in U$,
\item
$k(\a)$ for each $\a\in U$, 
\item
$k = \gcd_{\a\in U} \{ k(\a) \}$, 
\item
an integer $f(\a)$ for each $\a\in U$  
\end{itemize}
such that:
\begin{enumerate}
\item[{(a)}] 
for all $\a\in U$, 
all invertible $A$-lattices $L$ and all
$\beta\in (\a L)\smallsetminus \{ 0\}$ we have
$\langle \beta,\beta \rangle \ge (2^{n/2}+1)^2\cdot n$,
\item[(b)]
$k =  \sum_{\a\in U} f(\a)k(\a)$,
\item[(c)]
every prime divisor $\ell$ of $k$ satisfies $\ell \le c(n)$,
\item[(d)]
$\log_2(k) \le 2n$.
\end{enumerate}
\end{algorithm}

Steps:
\begin{enumerate}
\item
Run Algorithm \ref{usablesetsalg} to obtain a finite set $\SS$ that is usable for $A$.
\item
For each $S\in\SS \smallsetminus \{ S_0\}$, let $\a_S = \prod_{\p\in S} \p^{n(n+1)}$,
and put $\a_{S_0} = 2^{n+1}A$.
Output $U = \{ \a_S : S\in\SS\}$  
and the integers $k(\a)$ for each $\a\in U$.
\item
Use the extended Euclidean algorithm to compute $k$ and to find integers $f(\a)$
that satisfy (b).
\end{enumerate}

\begin{prop}
\label{Usetalgorpf}
Algorithm \ref{Usetalgor} is correct 
and runs in polynomial time. 
\end{prop} 

\begin{proof}
Since $\SS$ is usable, each $S\in\SS$ is vigilant.
It follows that each ideal $\a_S\in U$ is good. 
By Proposition \ref{xxlowerbd}(ii) we have (a).

Suppose $\ell$ is a prime number and $\ell > c(n)$.
Since $\SS$ is usable, there exists $S\in\SS$ such
that $\N(\p)\not\equiv 1 \bmod \ell$ and $\fchar(A/\p) \le b(n)$
for all $\p\in S$.
By Definition \ref{kadefn}, the positive integer
$k(\a_S)$ is not divisible by $\ell$.
Thus $k$ is not divisible by $\ell$, giving (c).

We have $k \le k(2^{n+1}A) \le 2^{2n}$,
giving (d).
\end{proof}

The following algorithm will be used in Algorithm \ref{connmainalgor}.
In the algorithm, the ideals $\a$ and $\b$ are the analogues of 
the prime numbers $m$ and $\ell$ 
of Algorithm 18.7 of \cite{LwS}, while $k(\a)$ is
the analogue of $k(m)$.

\begin{algorithm}
\label{ingredient4}
Given a connected CM-order $A$ of rank $n$ and an invertible $A$-lattice $L$,
the algorithm  either outputs ``$L$ has no short vector'' or
finds:
\begin{itemize}
\item
a positive integer $k$ each of whose 
prime factors is at most $c(n)$ and such that $\log_2(k) \le 2n$,
\item
an element $e_2\in L$ such that $Ae_2 + 2L = L$, and
\item
an element $s\in A/{2^{n+1}A}$ such that the coset $s\cdot (e_2^k + 2^{n+1}L^k) \in L^k/2^{n+1}L^k$
contains a short vector in $L^k$.
\end{itemize}

Steps:
\begin{enumerate}
\item
Apply Algorithm \ref{Usetalgor} to obtain  
a finite set
$U$ of good ideals $\a$ of $A$,
and $k(\a)$ and $f(\a)$ for each $\a\in U$, 
and $k = \gcd_{\a\in U} \{ k(\a) \}$
satisfying (a-d) of Algorithm \ref{Usetalgor}.
\item
Apply  
the algorithm associated to Proposition \ref{elemprop} 
to find $e_2\in L$ such that $Ae_2 + 2L = L$.
Let $\b = 2^{n+1}A$ 
and let $e_{\b}=e_2 + \b L \in L/\b L$.
\item
For each $\a\in U \smallsetminus \{\b\}$, do the following. 
Apply  
the algorithm associated to Proposition \ref{elemprop} 
to find $e_\a \in L/\a L$
such that $(A/\a)\cdot e_\a = L/\a L$.
Compute the $A$-lattice $L^{k(\a)}$ and the cosets
$e_\a^{k(\a)} \in L^{k(\a)}/\a L^{k(\a)}$ and
$e_{\b}^{k(\a)} \in L^{k(\a)}/{\b} L^{k(\a)}$.
Run Algorithm \ref{findingekmalg2}
to decide whether the coset $e_\a^{k(\a)}$ contains a vector $\nu_\a$
satisfying $\langle \nu_\a,\nu_\a\rangle = n$.
If no such $\nu_\a$ exists, 
terminate with ``no''. 
Otherwise, find
$s_\a\in (A/{\b})^\ast$ such that
$$
\nu_\a + {\b} L^{k(\a)} = s_\a\cdot e_{\b}^{k(\a)}
$$
and find a positive integer  $g(\a) \in f(\a)+\Z\cdot k(\b)$.
\item
Compute  
$$
s = 
\prod_{{\pile{\a \in U}{\a \neq \b}}} s_\a^{g(\a)} \in (A/{\b})^\ast.
$$
\item
Use  
the algorithm associated with Theorem \ref{LMmultalg}(iii)
to compute the $A$-lattice $L^k$
and the coset $e_{\b}^k\in L^k/{\b} L^k$.
\item
Compute $s\cdot e_{\b}^k = s(e_2^k+2^{n+1}L^k) \in L^k/{\b} L^k$. 
Apply Algorithm \ref{findingekmalg2} to compute
all $w\in s\cdot e_{\b}^k \subset L^k$ satisfying 
$w\bar{w}=1$. If there are none, output ``no''.
Otherwise, output $k$, $e_2$, and $s$.
\end{enumerate}
\end{algorithm}

\begin{prop}
\label{ingredient4pf}
Algorithm \ref{ingredient4} is correct and runs in  
polynomial time. 
\end{prop}

\begin{proof}
Each prime divisor of the positive integer $k$ output by 
Algorithm \ref{Usetalgor} is at most $c(n)$,
and $\log_2(k) \le 2n$. 

Since $L$ is invertible, by Corollary \ref{egothicacor}
the algorithm associated to Proposition \ref{elemprop} 
will find $e_2$ and $e_\a$ in Steps (ii) and (iii).
Since $L = Ae_2 + 2L$, it follows from Nakayama's Lemma that
$(A/\b)\cdot e_\b = L/\b L$, with  $e_\b$ defined as in Step (ii).

Take $z\in L$ with $z\bar{z} =1$.
Then $Az=L$ by Theorem \ref{shortthm}(ii).

Suppose $\a\in U$. 
Since 
$(A/\a)\cdot (z+\a L) = L/\a L=(A/\a)\cdot e_\a$, it follows that 
$z+\a L \in (A/\a)^\ast \cdot e_\a$.
By Lemma \ref{exponAa} we have
\begin{equation}
\label{xeaeqn}
z^{k(\a)}\in e_\a^{k(\a)}.
\end{equation}
Since $z^{k(\a)}\overline{z^{k(\a)}} = 1$, 
by Proposition \ref{shortequiv} we have
$\langle z^{k(\a)},z^{k(\a)}\rangle = n$.
Thus, Step (iii) will find a vector $\nu_\a$ for each $\a\neq\b$,
as long as $L$ has a short vector $z$.
The vector $\nu_\a$ is regular by Lemma \ref{reglem2} applied to
$L^{k(\a)}$ in place of $L$, 
and $\nu_\a$ is short 
by Proposition \ref{shortequiv}. 
We have $\nu_\a = z^{k(\a)}$ by the uniqueness property in 
Proposition \ref{findingekmprop2} 
(using property (a) of Algorithm \ref{Usetalgor}), and 
$$
z^{k(\a)}\bmod {\b} = \nu_\a \bmod {\b} = s_\a\cdot e_{\b}^{k(\a)}.
$$

Since $g(\a) \in f(\a)+\Z\cdot k(\b)$,
by Lemma \ref{exponAa} we have
$$
s = \prod_{{\b} \neq \a\in U} s_\a^{g(\a)} = \prod_{{\b} \neq \a\in U} s_\a^{f(\a)} \in (A/{\b})^\ast.
$$

Applying \eqref{xeaeqn} with $\a = {\b}$ gives
$z^{k({\b})}\bmod {\b} = 1\cdot e_{\b}^{k({\b})}\in\Lambda/{\b}\Lambda$.
Letting $s_\b = 1$, then 
\begin{multline*}
z^{k}\bmod {\b} = \prod_{\a\in U} (z^{k(\a)}\bmod {\b})^{f(\a)} 
= \prod_{\a\in U} \left(s_\a\cdot e_{\b}^{k(\a)}\right)^{f(\a)} \\
= \left(\prod_{\b\neq \a\in U} s_\a^{f(\a)}\right) e_{\b}^{k} 
= s\cdot e_{\b}^k
\in\Lambda/{\b}\Lambda,
\end{multline*}
so $z^{k}$  
is a short vector in the coset $s\cdot e_{\b}^k = s\cdot (e_2^k + 2^{n+1}L^k)$. 
Thus if $L$ has a short vector $z$, then Step (vi) outputs 
an element $s\in A/{2^{n+1}A}$ such that the coset $s\cdot (e_2^k + 2^{n+1}L^k)$
contains a short vector in $L^k$.

The algorithm runs in polynomial time since each step does.
\end{proof}

\section{Main algorithm}
\label{mainalgorsect}

Our main algorithm is Algorithm \ref{mainalgor}, which
first makes a reduction to the case of connected orders and
then calls on Algorithm \ref{connmainalgor}.

\begin{algorithm}
\label{connmainalgor}
Given a {\em connected} CM-order $A$  
and an invertible $A$-lattice $L$,
the algorithm decides whether or not $L$ is $A$-isomorphic
to the standard $A$-lattice, and if so, 
outputs a short vector $z\in L$
and an $A$-isomorphism $A \to L$ given by $a\mapsto a z$.

Steps:
\begin{enumerate}
\item
Apply Algorithm \ref{ingredient4}.  
If it outputs ``$L$ has no short vector'', terminate with ``no''.
Otherwise, Algorithm \ref{ingredient4} outputs $k$,  
$e_2$, and $s$.
Let $t_0=s$.
\item
Factor $k$. Let $p_1,\ldots,p_m$ be the prime divisors of $k$ with multiplicity,
and let $q_0 = k$. 
For $i=1,\ldots,m$ in succession, 
compute $q_i = q_{i-1}/p_i$, 
the lattice $L^{q_i}$, and 
the coset $e_2^{q_i} + 2^{n+1}L^{q_i} \in L^{q_i}/2^{n+1}L^{q_i}$,
and apply Algorithm \ref{almostmainalgor} 
where in place of inputs $L$, $r$, $\epsilon$, and $s$ one takes
$L^{q_i}$, $p_i$, 
$e_2^{q_i} + 2^{n+1}L^{q_i}$,
and $t_{i-1}$, respectively,
and where the output $t$ 
is called $t_i$.
If Algorithm \ref{almostmainalgor} ever outputs ``no'', terminate with ``no''.
\item
Otherwise output ``yes'', the short vector $z$ in the coset
$t_m\cdot (e_2 + 2^{n+1}L)$ where $t_m \in A/2^{n+1}A$ is the output 
of the last run of Algorithm \ref{almostmainalgor}, 
and the  
map $A \to L$, $a\mapsto a z$.
\end{enumerate}
\end{algorithm}

\begin{rem}
When we iterate Algorithm \ref{almostmainalgor} in Algorithm \ref{connmainalgor}, 
it often happens that we compute the same short vector in the same lattice twice,
namely in Step (v) of Algorithm \ref{almostmainalgor}
to compute $\alpha$ and then in Step (ii) 
of the next iteration of Algorithm \ref{almostmainalgor}
to compute $\nu$. 
However, that happens for two {\em different} representations of the same lattice, say  
$L^{h}$ and $(L^{h/p})^{\otimes p}$,
that are not easy to identify with each other
(but with the {\em same} $s \in A/2^{n+1}A$).
\end{rem}

\begin{rem}
The vector $z$ in Step (iii) of  Algorithm \ref{connmainalgor}
could be computed either using
Algorithm \ref{findingekmalg2}, or by taking $z=\alpha_m \in L$ where 
$\alpha_m$ is the element $\alpha$ computed in Step (iv) 
of the last run of Algorithm \ref{almostmainalgor}.
\end{rem}

\begin{prop}
\label{connmainalgorpf}
Algorithm \ref{connmainalgor} is correct and runs in  
polynomial time. 
\end{prop}

\begin{proof}
The $i$-th iteration of Step (ii) has as input 
the invertible $A$-lattice $L^{q_i} = L^{k/(p_1\cdots p_{i})}$,
and finds a coset containing a short vector in 
$L^{q_{i-1}} = L^{k/(p_1\cdots p_{i-1})}$,
as long as $L$ contains a short vector.
The output $z$, after $m$ iterations, is a short vector 
in the coset $t_m(e_2+2^{n+1}L)$. 

Recall that the size of the input describing $A$ is 
at least $n^3$.
Since each prime divisor of $k$ is at most $c(n)$ 
(as defined in Definition \ref{cbndefn}),
and  $\log_2(k) \le 2n$, one can
factor $k$ in polynomial time.
By Proposition \ref{almostmainalgorpf},
Algorithm \ref{almostmainalgor} runs in  
time at most polynomial in $r$ plus the length of the input. 
In Step (ii), in the $i$-th run of Algorithm \ref{almostmainalgor}
we have $r=p_i \le c(n)$.
It follows that Step (ii) runs in polynomial time. 
\end{proof}

The following result is Theorem 1.1 of \cite{RoU};
its associated algorithm is Algorithm 6.1 of \cite{RoU}.

\begin{prop}[\cite{RoU}]
\label{idempotsalg}
There is a deterministic polynomial-time algorithm that, given an order $A$,
lists all primitive idempotents of A.
\end{prop}

\begin{algorithm}
\label{mainalgor}
Given a CM-order $A$ 
and an  
$A$-lattice $L$,
the algorithm decides whether or not $L$ is $A$-isomorphic
to the standard $A$-lattice, and if so, 
outputs a short vector $z\in L$
and an $A$-isomorphism $A \to L$ given by $a\mapsto a z$.

Steps:
\begin{enumerate}
\item
Apply Algorithm \ref{elemprop2} to test $L$ for invertibility.
If $L$ is not invertible, terminate with ``no''. 
\item
Apply the algorithm from Proposition \ref{idempotsalg} to
compute the primitive idempotents of $A$, and apply
Lemma \ref{decomposeLA} to
decompose $A$ as a product of finitely many
connected rings $A = \prod_i A_i$
and
decompose $L$ as an orthogonal sum  $L=\bigperp_i L_i$
where $L_i$ is an invertible $A_i$-lattice.
\item
Apply Algorithm \ref{connmainalgor} to each 
$L_i$. 
If it ever terminates with ``no'', terminate with ``no $A$-isomorphism exists''.
Otherwise, it outputs maps $A_i \to L_i$, $a\mapsto a z_i$ for each $i$.
Output ``yes'', $z=(z_i)_i \in A = \prod_i A_i$, and the map 
$A = \prod_i A_i \to L=\bigperp_i L_i$, 
$(a_i)_i \mapsto (a_iz_i)_i$.
\end{enumerate}
\end{algorithm}

Proposition \ref{invertequivcor}(iii) now enables us to convert Algorithm \ref{mainalgor}
into an algorithm to test whether two $A$-lattices are $A$-isomorphic (and
produce an isomorphism). This is our analogue of Algorithm 19.4 of
\cite{LwS}.

\begin{algorithm}
\label{algorLM}
Given a CM-order $A$
and invertible $A$-lattices $L$ and $M$, the algorithm
decides whether or not $L$
and $M$ are isomorphic as $A$-lattices, and if so, gives 
such an $A$-isomorphism.
\begin{enumerate}

\item 
Compute $L \otimes_{A} \overline{M}$.
\item
Apply Algorithm \ref{mainalgor} to find   
an $A$-isomorphism $A \isom L \otimes_{A} \overline{M},$ or a proof that none exists. In the latter case, terminate with ``no''.
\item
Using this map and the map 
$\overline{M} \otimes_{A} M \to A$, 
$\overline{y} \otimes x \mapsto \overline{y} \cdot x,$
output the composition of the (natural) maps
$$M \isom A \otimes_{A} M \isom
L \otimes_{A} \overline{M}\otimes_{A} M
\isom L \otimes_{A} A \isom L.$$
\end{enumerate}
\end{algorithm}

It is clear that 
Algorithms \ref{mainalgor} and \ref{algorLM} are correct and run in  
polynomial time. 
Theorems \ref{mainIwthm} and \ref{mainIwthmcor} now follow from 
Algorithms \ref{mainalgor} and \ref{algorLM}
and Theorem \ref{I1I2isom}.

\end{document}